\documentclass[leqno]{siamltex}
\usepackage{multirow}
\usepackage{verbatim} 
\usepackage{soul} 
\usepackage{url} 
\usepackage{graphicx,amsmath,amsfonts,amssymb,subfigure}
\usepackage[usenames,dvipsnames]{xcolor}
\usepackage{algorithmic}
\usepackage{tikz}
\usetikzlibrary{arrows,patterns}
\usepackage{mathtools}
\usepackage{booktabs}
\usepackage{textcomp}
\usepackage{setspace}
\usepackage{color}
\usepackage{enumerate}
 \usepackage{pgfplots}
 \usetikzlibrary{external}
 \usetikzlibrary{patterns}
 \usetikzlibrary{matrix}

\RequirePackage{xcolor}

\graphicspath{{./FIGS/}}
\newcommand{\argmin}{\operatornamewithlimits{arg\,min}}

\newcommand{\xpmatrix}[1]{\begin{pmatrix} #1 \end{pmatrix}}

\newcommand{\ra}[1]{\renewcommand{\arraystretch}{1}}

\newcommand{\beeq}[1]{\begin{equation}\label{#1}}
\newcommand{\eneq}{\end{equation}}

\newtheorem{algor}{{\sc Algorithm}}[section]

\def\betab{\begin{tabbing} 
xxxx\=xxxx\=xxx\=xx\=xx\=xx\=xx\=xx\=xx\=xx\=xx\=xx\=xx\= \kill} 
\def\entab{\end{tabbing}\vspace{-0.12in}}
 
\newenvironment{algorithm0}{\begin{algor} \sl }{ \end{algor} }

\title{Beyond AMLS: Domain decomposition with rational filtering\thanks{This work supported 
by NSF under award CCF-1505970. 
}}

\author{ vassilis kalantzis\thanks{Address: Computer Science \& Engineering, University of Minnesota, Twin Cities. 
{\tt \{kalan019,yxi,saad\} @umn.edu} } \and yuanzhe xi\footnotemark[2] \and yousef saad\footnotemark[2] }

\begin{document}

\maketitle

\begin{abstract}
This paper proposes a rational filtering domain decomposition technique for the solution of 
large and sparse symmetric generalized eigenvalue problems. The proposed technique is purely 
algebraic and decomposes the eigenvalue problem associated with each subdomain into two 
disjoint subproblems. The first subproblem is associated with the interface variables and 
accounts for the interaction among neighboring subdomains. To compute the solution of the 
original eigenvalue problem at the interface variables we leverage ideas from contour integral 
eigenvalue solvers. The second subproblem is associated with the interior variables in each 
subdomain and can be solved in parallel among the different subdomains using real arithmetic 
only. Compared to rational filtering projection methods applied to the original matrix pencil, 
the proposed technique integrates only a part of the matrix resolvent while it applies any 
orthogonalization necessary to vectors whose length is equal to the number of interface variables. 
In addition, no estimation of the number of eigenvalues lying inside the interval of interest 
is needed. Numerical experiments performed in distributed memory architectures illustrate the 
competitiveness of the proposed technique against rational filtering Krylov approaches.
\end{abstract}



\section{Introduction}

The typical approach to solve large and sparse symmetric eigenvalue problems of the form 
$Ax=\lambda M x$ is via a Rayleigh-Ritz (projection) process on a low-dimensional subspace
that spans an invariant subspace associated with the $nev\geq1$ eigenvalues of interest, 
e.g., those located inside the real interval $[\alpha,\beta]$. 



One of the main bottlenecks of Krylov projection methods in large-scale eigenvalue computations 
is the cost to maintain the orthonormality of the basis of the Krylov subspace; especially when 
$nev$ runs in the order of hundreds or thousands. To reduce the orthonormalization and memory 
costs, it is typical to enhance the convergence rate of the Krylov projection method of choice 
by a filtering acceleration technique so that eigenvalues located outside the interval of interest 
are damped to (approximately) zero. For generalized eigenvalue problems, a standard choice is to 
exploit rational filtering techniques, i.e., to transform the original matrix pencil into a 
complex, rational matrix-valued function \cite{VanBarel2016526,FEAST,pfeast,doi:10.1137/13090866X,feastnonsym,SaSu,
doi:10.1137/16M1061965,VanBarel2015}. While rational filtering approaches reduce orthonormalization 
costs, their main bottleneck is the application the transformed pencil, i.e., the solution of the 
associated linear systems.

An alternative to reduce the computational costs (especially that of orthogonalization) 
in large-scale eigenvalue computations is to consider domain decomposition-type 
approaches (we refer to \cite{Smith:1996:DDP:238150,toselli2005domain} for an in-depth 
discussion of domain decomposition). Domain decomposition decouples the original 
eigenvalue problem into two separate subproblems; one defined locally in the interior 
of each subdomain, and one defined on the interface region connecting neighboring 
subdomains. Once the original eigenvalue problem is solved for the interface region, the 
solution associated with the interior of each subdomain is computed independently of the 
other subdomains \cite{Philippe20071471,AMLS,AMLS2,Lui200017,Lui199835,DDSpec,10.2307/2158123}.
When the number of variables associated with the interface region is much smaller than 
the number of global variables, domain decomposition approaches can provide approximations 
to thousands of eigenpairs while avoiding excessive orthogonalization costs. One 
prominent such example is the Automated Multi-Level Substructuring (AMLS) method 
\cite{AMLS,AMLS2,Gao:2008:IEA:1377596.1377600}, originally developed by the structural 
engineering community for the frequency response analysis of Finite Element automobile 
bodies. AMLS has been shown to be considerably faster than the NASTRAN industrial package 
\cite{KOMZSIK1991167} in applications where $nev \gg 1$. However, the accuracy provided 
by AMLS is good typically only for eigenvalues that are located close to a user-given 
real-valued shift \cite{AMLS2}.



In this paper we describe the Rational Filtering Domain Decomposition Eigenvalue Solver (RF-DDES), 
an approach which combines domain decomposition with rational filtering. Below, we list the main 
characteristics of RF-DDES:

1) \emph{Reduced complex arithmetic and orthgonalization costs.} Standard rational filtering 
techniques apply the rational filter to the entire matrix pencil, i.e., they require the solution of 
linear systems with complex coefficient matrices of the form $A-\zeta M$ for different values 
of $\zeta$. In contrast, RF-DDES applies the rational filter only to that part of $A-\zeta M$ 
that is associated with the interface variables. As we show later, this approach has several 
advantages: a) if a Krylov projection method is applied, orthonormalization needs to be applied 
to vectors whose length is equal to the number of interface variables only, b) while RF-DDES 
also requires the solution of complex linear systems, the associated computational cost is 
lower than that of standard rational filtering approaches, c) focusing on the interface 
variables only makes it possible to achieve convergence of the Krylov projection method in even 
fewer than $nev$ iterations. In contrast, any Krylov projection method applied to a rational 
transformation of the original matrix pencil must perform at least $nev$ iterations.

2) \emph{Controllable approximation accuracy.}
Domain decomposition approaches like AMLS might fail to provide high accuracy for all eigenvalues 
located inside $[\alpha,\beta]$. This is because AMLS solves only an approximation of the original 
eigenvalue problem associated with the interface variables of the domain. In contrast, RF-DDES can 
compute the part of the solution associated with the interface variables highly accurately. As a
result, if not satisfied with the accuracy obtained by RF-DDES, one can simply refine the part of
the solution that is associated with the interior variables.

3) \emph{Multilevel parallelism.}
The solution of the original eigenvalue problem associated with the interior variables of each 
subdomain can be applied independently in each subdomain, and requires only real arithmetic.
Moreover, being a combination of domain decomposition and rational filtering techniques, RF-DDES 
can take advantage of different levels of parallelism, making itself appealing for execution 
in high-end computers. We report results of experiments performed in distributed memory environments 
and verify the effectiveness of RF-DDES.

Throughout this paper we are interested in computing the $nev \geq 1$ eigenpairs $(\lambda_i,x^{(i)})$ 
of $Ax^{(i)}=\lambda_iMx^{(i)},\ i=1,\ldots,n$, for which $\lambda_i \in [\alpha,\beta]$, $\alpha \in 
\mathbb{R},\ \beta \in \mathbb{R}$. The $n\times n$ matrices $A$ and $M$ are assumed large, sparse and 
symmetric while $M$ is also positive-definite (SPD). 
For brevity, we will refer to the linear SPD matrix pencil $A-\lambda M$ simply as $(A,M)$.

The structure of this paper is as follows: Section \ref{section2} describes the general working of 
rational filtering and domain decomposition eigenvalue solvers. Section \ref{section3} describes 
computational approaches for the solution of the eigenvalue problem associated with the interface 
variables. Section \ref{interior_part} describes the solution of the original eigenvalue problem 
associated with the interior variables in each subdomain. Section \ref{main_algorithm} combines all 
previous discussion into the form of an algorithm. Section \ref{numerical_experiments} presents  
experiments performed on model and general matrix pencils. Finally, Section \ref{conclusion} contains our concluding remarks.

\section{Rational filtering and domain decomposition eigensolvers} \label{section2}

In this section we review the concept of rational filtering for the solution of real 
symmetric generalized eigenvalue problems. In addition, we present a prototype 
Krylov-based rational filtering approach to serve as a baseline algorithm, while also 
discuss the solution of symmetric generalized eigenvalue problems from a domain 
decomposition viewpoint.

Throughout the rest of this paper we will assume that the eigenvalues of $(A,M)$ 
are ordered so that eigenvalues $\lambda_1,\ldots,\lambda_{nev}$ are located within
$[\alpha,\beta]$ while eigenvalues $\lambda_{nev+1},\ldots,\lambda_{n}$ are located
outside $[\alpha,\beta]$.

\subsection{Construction of the rational filter}     \label{filter_func}

The classic approach to construct a rational filter function $\rho(\zeta)$ is to exploit 
the Cauchy integral representation of the indicator function $I_{[\alpha,\beta]}$, where  
$I_{[\alpha,\beta]}(\zeta)=1,\ {\rm iff}\ \zeta \in[\alpha,\beta]$, and $I_{[\alpha,\beta]}
(\zeta)=0,\ {\rm iff}\ \zeta \notin[\alpha,\beta]$; see the related discussion in 
\cite{VanBarel2016526,FEAST,pfeast,doi:10.1137/13090866X,feastnonsym,SaSu,doi:10.1137/16M1061965,
VanBarel2015} (see also \cite{doi:10.1137/140980090,doi:10.1137/16M1061965,winkelmann2017non,
doi:10.1137/140984129} for other filter functions not based on Cauchy's formula). 

Let $\Gamma_{[\alpha,\beta]}$ be a smooth, closed contour that encloses only those 
$nev$ eigenvalues of $(A,M)$ which are located inside $[\alpha,\beta]$, e.g., a 
circle centered at $(\alpha+\beta)/2$ with radius $(\beta-\alpha)/2$. We then have
\begin{equation} \label{eq:contour}
I_{[\alpha,\beta]}(\zeta) = \dfrac{-1}{2\pi i}\int_{\Gamma_{[\alpha,\beta]}} \dfrac{1}{\zeta-\nu} d\nu,
\end{equation}
where the integration is performed counter-clockwise. 
The filter function ${\rho}(\zeta)$ can be obtained by applying a quadrature rule to discretize 
the right-hand side in (\ref{eq:contour}):
\begin{equation} \label{eq:contour2}
{\rho}(\zeta) = \sum_{\ell=1}^{2N_c} \frac{\omega_{\ell}}{\zeta -\zeta_{\ell}}, 
\end{equation}
where $\{\zeta_{\ell},\omega_{\ell}\}_{1\leq \ell\leq 2N_c}$ are the poles and weights of the 
quadrature rule. If the $2N_c$ poles in (\ref{eq:contour2}) come in conjugate pairs, and the 
first $N_c$ poles lie on the upper half plane, (\ref{eq:contour2}) can be simplified into
\begin{equation} \label{eq:contour3}
{\rho}(\zeta) = 2\Re e \left \lbrace \sum_{\ell=1}^{N_c} \frac{\omega_{\ell}}{\zeta -\zeta_{\ell}} \right \rbrace, \ \text{when} \ \ \zeta \in \mathbb{R}.
\end{equation}

\begin{figure} 
\centering
\includegraphics[width=0.49\textwidth]{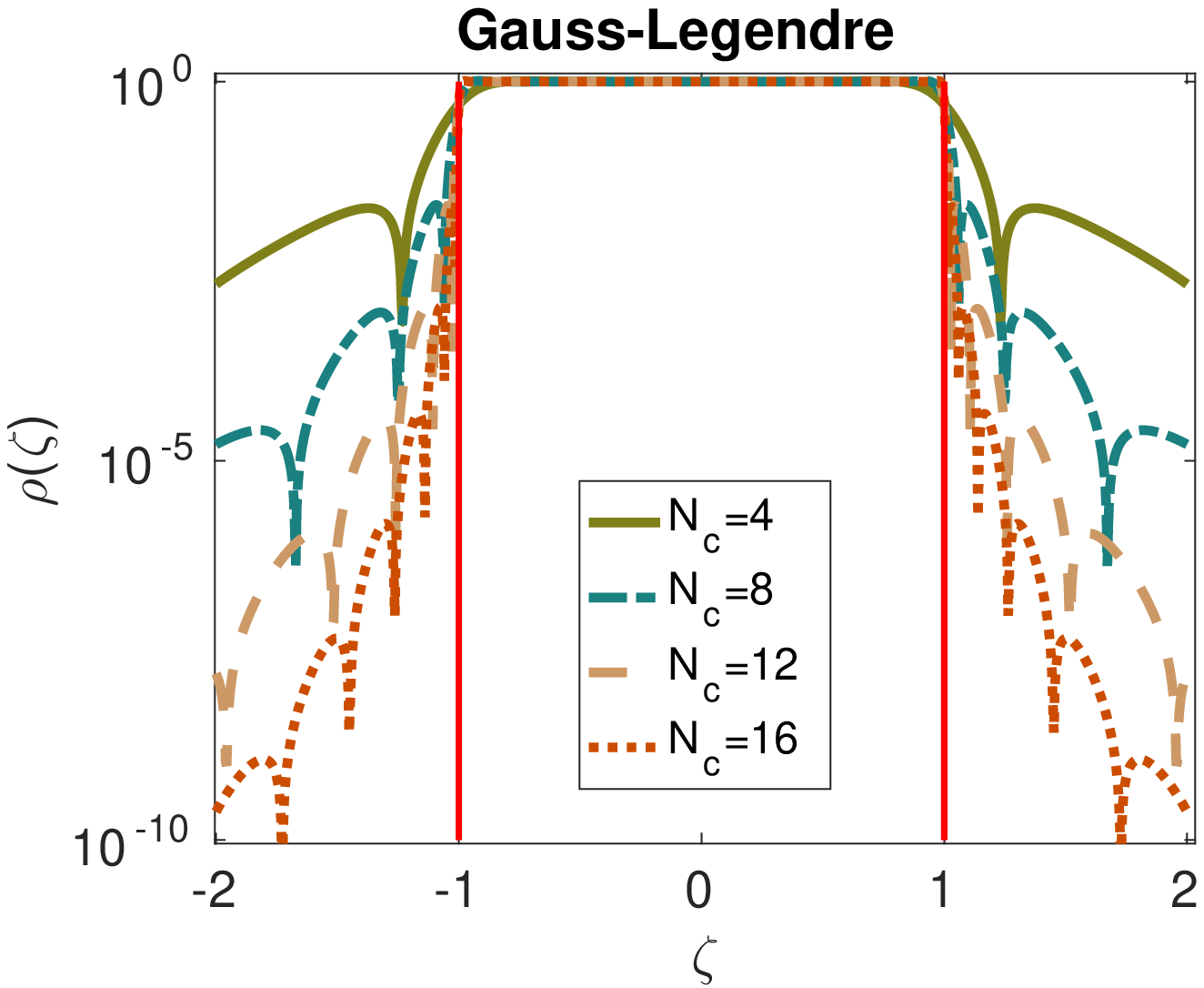}
\includegraphics[width=0.49\textwidth]{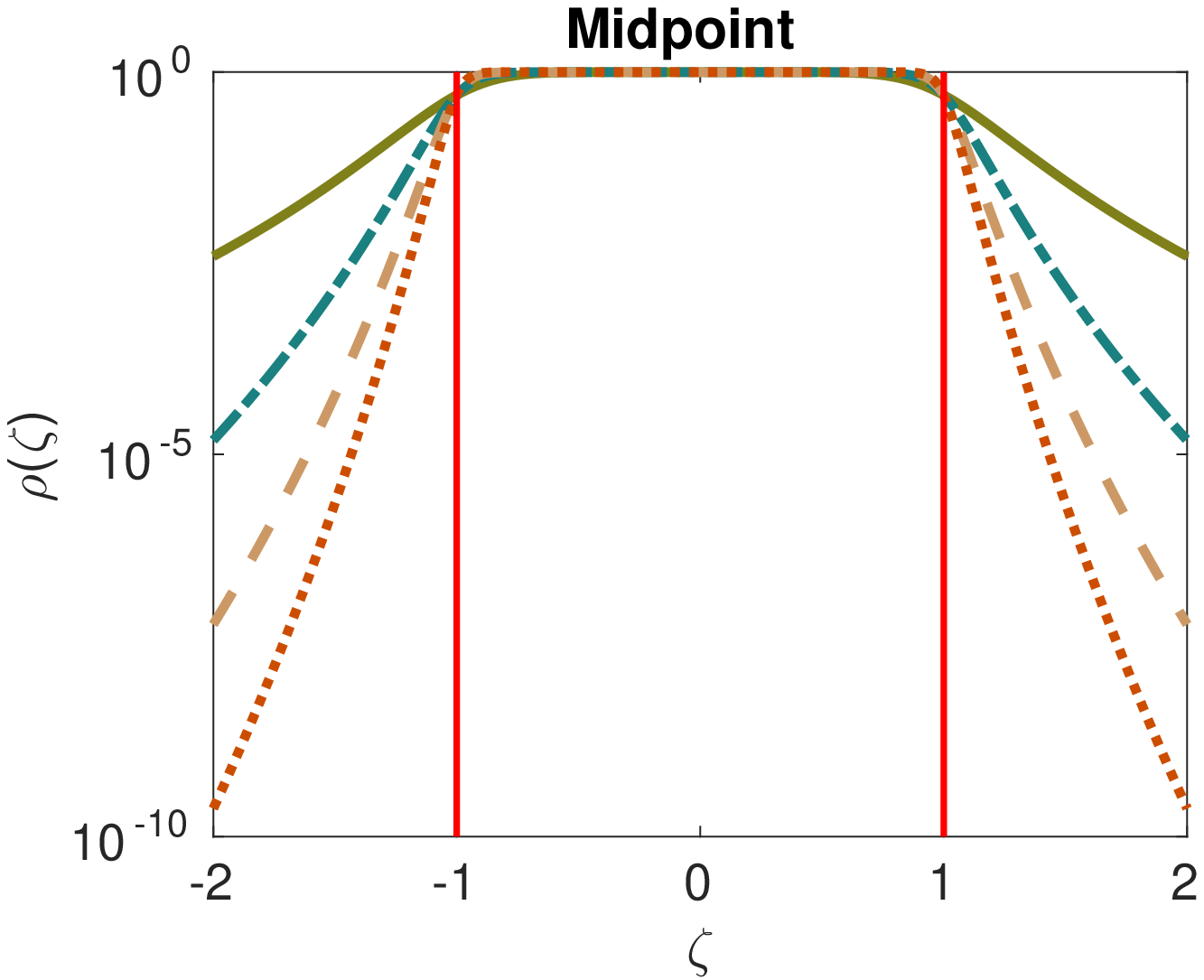}
\caption{The modulus of the rational filter function $\rho(\zeta)$ when $\zeta \in [-2,2]$.
Left: Gauss-Legendre rule. Right: Midpoint rule.\label{fig:filter_value}}
\end{figure}

Figure \ref{fig:filter_value} plots the modulus of the rational function $\rho(\zeta)$ 
(scaled such that $\rho(\alpha) \equiv \rho(\beta) \equiv 1/2$) in the interval $\zeta 
\in [-2,2]$, where $\{\zeta_{\ell},\omega_{\ell}\}_{1\leq \ell \leq 2N_c}$ are obtained 
by numerically approximating $I_{[-1,1]}(\zeta)$ by the Gauss-Legendre rule (left) and 
Midpoint rule (right). Notice that as $N_c$ increases, $\rho(\zeta)$ becomes a more 
accurate approximation of $I_{[-1,1]}(\zeta)$ \cite{doi:10.1137/16M1061965}. Throughout
the rest of this paper, we will only consider the Midpoint rule \cite{Abramowitz:1974:HMF:1098650}.

\subsection{Rational filtered Arnoldi procedure}

Now, consider the rational matrix function $\rho(M^{-1}A)$ with $\rho(.)$ defined as in 
(\ref{eq:contour2}):
\begin{equation} \label{rhoAM}
\rho(M^{-1}A) =  2\Re e \left \lbrace \sum_{\ell=1}^{N_c} \omega_{\ell} (A-\zeta_{\ell} M)^{-1}M \right \rbrace. 
\end{equation}
The eigenvectors of $\rho(M^{-1}A)$ are identical to those of $(A,M)$, while the corresponding
eigenvalues are transformed to $\left\{\rho(\lambda_j)\right\}_{j=1,\ldots,
n}$. Since $\rho(\lambda_{1}),\ldots,\rho(\lambda_{nev})$ are all larger than $\rho(\lambda_{nev+1}),\ldots,\rho(\lambda_{n})$, 
the eigenvalues of $(A,M)$ located inside $[\alpha,\beta]$ become the dominant ones in 
$\rho(M^{-1}A)$. Applying a projection method to $\rho(M^{-1}A)$ can then lead to fast 
convergence towards an invariant subspace associated with the eigenvalues of $(A,M)$
located inside $[\alpha,\beta]$.

One popular choice as the projection method in rational filtering approaches is that of 
subspace iteration, e.g., as in the FEAST package \cite{FEAST,pfeast,feastnonsym}. One 
issue with this choice is that an estimation of $nev$ needs be provided in order to 
determine the dimension of the starting subspace. In this paper, we exploit Krylov subspace 
methods to avoid the requirement of providing an estimation of $nev$.

\vspace*{2mm}
\begin{minipage}{\dimexpr\linewidth-9\fboxsep-9\fboxrule\relax}
\vbox{
\begin{algorithm0}{RF-KRYLOV} \label{alg:arnoldi_filter}
\betab
\>0. \> Start with $q^{(1)} \in \mathbb{R}^n\ s.t.\ \|q^{(1)}\|_2 = 1$    \\
\>1. \> For $\mu=1,2,\ldots$                                                                  \\
\>2. \> \> Compute $w= 2\Re e \left\lbrace \sum_{\ell=1}^{N_c}\omega_{\ell} (A-\zeta_{\ell} M)^{-1}Mq^{(\mu)} \right\rbrace$   \\
\>3. \> \> For $\kappa=1,\ldots,\mu$                                                          \\
\>4. \> \> \> $h_{\kappa,\mu}=w^T q^{(\kappa)}$                                                  \\
\>5. \> \> \> $w = w - h_{\kappa,\mu}q^{(\kappa)}$                                   \\
\>6. \> \> End                                                                                \\
\>7. \> \> $h_{\mu+1,\mu} = \|w\|_2$                  \\
\>8. \> \> If $h_{\mu+1,\mu} =0$\\
\>9. \>\>\> generate a unit-norm $q^{(\mu+1)}$ orthogonal to $q^{(1)},\ldots, q^{(\mu)}$\\
\>10. \>\> Else\\
\>11. \> \> \> $q^{(\mu+1)} = w / h_{\mu+1,\mu}$   \\
\>12 \> \> EndIf  \\
\>13. \> \> If the sum of eigenvalues of $H_{\mu}$ no less than $1/2$ is unchanged\\
\> \> \> during the last few iterations; BREAK; EndIf                                              \\
\>14.\> End                                                                                   \\
\>15. \> Compute the eigenvectors of $H_{\mu}$ and form the Ritz vectors of $(A,M)$ \\
\>16. \> For each Ritz vector $\hat{q}$, compute the corresponding approximate \\ \> \> Ritz value as the Rayleigh quotient $\hat{q}^TA\hat{q}/\hat{q}^TM\hat{q}$\\
\entab
\end{algorithm0}
}
\end{minipage}
\vspace*{1mm}

Algorithm \ref{alg:arnoldi_filter} sketches the Arnoldi procedure applied to $\rho(M^{-1}A)$ 
for the computation of all eigenvalues of $(A,M)$ located inside $[\alpha,\beta]$ and
associated eigenvectors. Line 2 computes the ``filtered'' vector $w$ by applying 
the matrix function $\rho(M^{-1}A)$ to $q^{(\mu)}$, which in turn requires the solution of 
the $N_c$ linear systems associated with matrices $A-\zeta_{\ell}M,\ \ell=1,\ldots,N_c$. 
Lines 3-12 orthonormalize $w$ against the previous Arnoldi vectors $q^{(1)},\ldots,q^{(\mu)}$ 
to produce the next Arnoldi vector $q^{(\mu+1)}$. Line 13 checks the sum of those eigenvalues 
of the upper-Hessenberg matrix $H_\mu$ which are no less than $1/2$. If this sum remains constant 
up to a certain tolerance, the outer loop stops. Finally, line 16 computes the Rayleigh quotients 
associated with the approximate eigenvectors of $\rho(M^{-1}A)$ (the Ritz vectors obtained in
line 15). 

Throughout the rest of this paper, Algorithm \ref{alg:arnoldi_filter} will be abbreviated 
as RF-KRYLOV. 


\subsection{Domain decomposition framework} \label{ddf}

Domain decomposition eigenvalue solvers \cite{ddfeast,DDSpec,AMLS,AMLS2} compute spectral 
information of $(A,M)$ by decoupling the original eigenvalue problem into two separate 
subproblems: one defined locally in the interior of each subdomain, and one restricted to 
the interface region connecting neighboring subdomains. Algebraic domain decomposition 
eigensolvers start by calling a graph partitioner \cite{SCOTCH,METIS-SIAM} to decompose the 
adjacency graph of $|A|+|M|$ into $p$ non-overlapping subdomains. If we then order the 
interior variables in each subdomain before the interface variables across all subdomains, 
matrices ${A}$ and ${M}$ then take the following block structures: 
\begin{equation} \label{dd1}
A = 
\xpmatrix{
B_1      &        &           &          &  E_{1 }  \cr
         & B_2    &           &          &  E_{2 }  \cr
         &        &   \ddots  &          &  \vdots  \cr
         &        &           &  B_p     &   E_p    \cr
 E_1^T   & E_2^T  &   \ldots  &  E_p^T   &   C      \cr
},\ \ \ 
M = 
\xpmatrix{
M_B^{(1)}   &           &           &             &  M_E^{(1)}  \cr
            & M_B^{(2)} &           &             &  M_E^{(2)}  \cr
            &           &   \ddots  &             &  \vdots     \cr
            &           &           &  M_B^{(p)}  &  M_E^{(p)}  \cr
 M_E^{(1)T} & M_E^{(2)T}&   \ldots  &  M_E^{(p)T} &   M_C       \cr
}.
\end{equation}
If we denote the number of interior and interface variables lying in the $j$th subdomain 
by $d_j$ and $s_j$, respectively, and set $s=\sum_{j=1}^p s_j$, then $B_j$ and $M_B^{(j)}$ 
are square matrices of size $d_j\times d_j$, $E_j$ and $M_E^{(j)}$ are rectangular matrices 
of size $d_j\times s_j$, and $C$ and $M_C$ are square matrices of size $s\times s$. Matrices 
$E_i,\ M_E^{(j)}$ have a special nonzero pattern of the form $E_j=[0_{d_j,\ell_j},\hat{E}_j,
0_{d_j,\nu_j}]$, and $M_E^{(j)}=[0_{d_j,\ell_j},\hat{M}_E^{(j)},0_{d_j,\nu_j}]$, where 
$\ell_j=\sum_{k=1}^{k<j}s_k$, $\nu_j = \sum_{k>j}^{k=p}s_k$, and $0_{\chi,\psi}$ denotes 
the zero matrix of size $\chi \times \psi$.

Under the permutation (\ref{dd1}), $A$ and $M$ can be also written in a compact form as:
\beeq{eq1}
A = 
\begin{pmatrix} 
B     &  E \cr    E^T   &  C 
\end{pmatrix},\ \ \ 
M = 
\begin{pmatrix} 
M_B   &  M_E \cr  M_E^T & M_C 
\end{pmatrix}.
\eneq
The block-diagonal matrices $B$ and $M_B$ are of size $d\times d$, where $d=\sum_{i=1}^p d_i$, 
while $E$ and $M_E$ are of size $d\times s$.

\subsubsection{Invariant subspaces from a Schur complement viewpoint}

Domain decomposition eigenvalue solvers decompose the construction of the Rayleigh-Ritz 
projection subspace ${\cal Z}$ is formed by two separate parts. More specifically, ${\cal 
Z}$ can be written as
\begin{equation} \label{dd_rr}
{\cal Z} = {\cal U} \oplus {\cal Y},
\end{equation}
where ${\cal U}$ and ${\cal Y}$ are subspaces that are orthogonal to each other and approximate 
the part of the solution associated with the interior and interface variables, respectively.

Let the $i$th eigenvector of $(A,M)$ be partitioned as 
\begin{equation} \label{eig_part}
x^{(i)}=
 \begin{pmatrix}
   u^{(i)}                                     \\[0.3em]
   y^{(i)}                                     \\[0.3em]
  \end{pmatrix},\ i=1,\ldots,n,
\end{equation}
where $u^{(i)} \in \mathbb{R}^{d}$ and $y^{(i)}\in \mathbb{R}^s$ correspond to the eigenvector 
part associated with the interior and interface variables, respectively. We can then rewrite 
$Ax^{(i)}=\lambda_i Mx^{(i)}$ in the following block form
\begin{equation} \label{schur_eigvec1}
  \begin{pmatrix}
   B-\lambda_i M_B     &  E-\lambda_i M_E      \\[0.3em]
   E^T-\lambda_i M_E^T &  C-\lambda_i M_C      \\[0.3em]
  \end{pmatrix}
  \begin{pmatrix}
   u^{(i)}                                     \\[0.3em]
   y^{(i)}                                     \\[0.3em]
  \end{pmatrix}=0.
\end{equation}
Eliminating $u^{(i)}$ from the second equation in (\ref{schur_eigvec1}) leads to the 
following nonlinear eigenvalue problem of size $s\times s$:
\begin{equation} \label{schur_eigvec3}
 [C-\lambda_i M_C -(E-\lambda_i M_E)^T(B-\lambda_i M_B)^{-1}(E-\lambda_i M_E)]y^{(i)} = 0.
\end{equation}
Once $y^{(i)}$ is computed in the above equation, $u^{(i)}$ can be recovered by the following linear 
system solution
\begin{equation} \label{schur_eigvec2}
 (B-\lambda_i M_B)u^{(i)} = - (E-\lambda_i M_E)y^{(i)}.
\end{equation}

In practice, since matrices $B$ and $M_B$ in (\ref{dd1}) are block-diagonal, the $p$ sub-vectors 
$u_j^{(i)}\in \mathbb{R}^{d_j}$ of $u^{(i)}=[(u_1^{(i)})^T,\ldots,(u_p^{(i)})^T]^T$ can 
be computed in a decoupled fashion among the $p$ subdomains as
\begin{equation} \label{compute_u_parallel}
 (B_j-\lambda_i M_B^{(j)}) u_j^{(i)} = -(\hat{E}_j-\lambda_i \hat{M}_E^{(j)})y_j^{(i)},\  j=1,\ldots,p,
\end{equation}
where $y_j^{(i)} \in \mathbb{R}^{s_j}$ is the subvector of $y^{(i)}=[(y_1^{(i)})^T,\ldots,
(y_p^{(i)})^T]^T$ that corresponds to the $j$th subdomain. 

By (\ref{schur_eigvec3}) and (\ref{schur_eigvec2}) we see that the subspaces ${\cal U}$ and 
${\cal Y}$ in (\ref{dd_rr}) should ideally be chosen as
\small
 \begin{align}
  & {\cal Y}= {\rm span} \left\lbrace \left[y^{(1)},\ldots,y^{(nev)}\right] \right\rbrace,   \\
  & {\cal U}= {\rm span} \left\lbrace \left[(B-\lambda_1 M_B)^{-1} (E-\lambda_1 M_E)y^{(1)},\ldots,(B-\lambda_{nev} M_B)^{-1} (E-\lambda_{nev} M_E)y^{(nev)}\right] \right\rbrace.
 \end{align}
\normalsize
The following two sections propose efficient numerical schemes to approximate these two 
subspaces. 

\section{Approximation of $\operatorname{span}\{y^{(1)},\ldots,y^{(nev)}\}$} \label{section3}

In this section we propose a numerical scheme to approximate $\operatorname{span}
\{y^{(1)},\ldots,y^{(nev)}\}$.

\subsection{Rational filtering restricted to the interface region} \label{interf2}

Let us define the following matrices:
\begin{equation*}
\begin{aligned}
 & B_{\zeta_\ell}=B-\zeta_\ell M_B, & E_{\zeta_\ell}=E-\zeta_\ell M_E, & \ \ C_{\zeta_\ell}=C-\zeta_\ell M_C.
\end{aligned}
\end{equation*}
Then, each matrix $(A-\zeta_\ell M)^{-1}$ in (\ref{rhoAM}) can be expressed as
\begin{equation} \label{inv_dd2}
  (A-\zeta_\ell M)^{-1}=
  \begin{pmatrix}
   B_{\zeta_\ell}^{-1} + B_{\zeta_\ell}^{-1}E_{\zeta_\ell}S_{\zeta_\ell}^{-1}E_{\zeta_\ell}^HB_{\zeta_\ell}^{-1}  &   -B_{\zeta_\ell}^{-1}E_{\zeta_\ell}S_{\zeta_\ell}^{-1}    \\[0.3em]
   -S_{\zeta_\ell}^{-1}E_{\zeta_\ell}^HB_{\zeta_\ell}^{-1}                                                        &   S_{\zeta_\ell}^{-1}                                      \\[0.3em]
  \end{pmatrix},
\end{equation}
where
\begin{equation} S_{\zeta_\ell} = C_{\zeta_\ell} -E_{\zeta_\ell}^HB_{\zeta_\ell}^{-1}E_{\zeta_\ell}, 
\end{equation}
denotes the corresponding Schur complement matrix.

Substituting (\ref{inv_dd2}) into (\ref{rhoAM}) leads to 
\begin{align} \label{num_cont_big}
\rho(M^{-1}A) = 
 & 2\Re e \left \lbrace \sum_{\ell=1}^{N_c} \omega_{\ell}
\begin{bmatrix}
   \left[B_{\zeta_{\ell}}^{-1} + B_{\zeta_{\ell}}^{-1}E_{\zeta_{\ell}}S_{\zeta_\ell}^{-1}E_{\zeta_{\ell}}^{H}B_{\zeta_{\ell}}^{-1}\right] & -  B_{\zeta_{\ell}}^{-1}E_{\zeta_{\ell}}S_{\zeta_\ell}^{-1} \\[0.3em]
 - S_{\zeta_\ell}^{-1}E_{\zeta_{\ell}}^{H}B_{\zeta_{\ell}}^{-1}                                                     &  S_{\zeta_\ell}^{-1}                        \\[0.3em]
\end{bmatrix}\right \rbrace M.
\end{align}
On the other hand, we have for any $\zeta \notin \Lambda(A,M)$:
\begin{equation}(A-\zeta M)^{-1} = \sum_{i=1}^n \dfrac{x^{(i)} (x^{(i)})^T}{\lambda_i-\zeta}. \label{pole_eig}\end{equation}
The above equality yields another expression for $\rho(M^{-1}A)$:
\begin{align} \label{lala}
\rho(M^{-1}A) & = 
\sum_{i=1}^n \rho(\lambda_i) x^{(i)} (x^{(i)})^TM    \\ 
& = \sum_{i=1}^n \rho(\lambda_i)
\begin{bmatrix} \label{lala2}
   u^{(i)} (u^{(i)})^T  &  u^{(i)} (y^{(i)})^T       \\[0.3em] 
   y^{(i)} (u^{(i)})^T  &  y^{(i)} (y^{(i)})^T
\end{bmatrix}M.
\end{align}

Equating the (2,2) blocks of the right-hand sides in (\ref{num_cont_big}) and (\ref{lala2}), 
yields
\begin{equation} \label{finalizing}
 2\Re e \left \lbrace \sum_{\ell=1}^{N_c} \omega_{\ell}S_{\zeta_{\ell}}^{-1} \right \rbrace = \sum_{i=1}^n \rho(\lambda_i) y^{(i)} (y^{(i)})^T.
\end{equation}
Equation (\ref{finalizing}) provides a way to approximate $\operatorname{span}\{y^{(1)},\ldots,
y^{(nev)}\}$ through the information in $S_{\zeta_{\ell}}^{-1}$. The coefficient $\rho(\lambda_i)$ 
can be interpreted as the contribution of the direction $y^{(i)}$ in $2\Re e \left \lbrace 
\sum_{\ell=1}^{N_c} \omega_{\ell}S_{\zeta_{\ell}}^{-1} \right \rbrace$. In the ideal case where 
$\rho(\zeta)\equiv \pm I_{[\alpha,\beta]}(\zeta)$, we have $\sum_{i=1}^n \rho(\lambda_i) y^{(i)} 
(y^{(i)})^T = \pm \sum_{i=1}^{nev} y^{(i)} (y^{(i)})^T$. In practice, $\rho(\zeta)$ will only 
be an approximation to $\pm I_{[\alpha,\beta]}(\zeta)$, and since $\rho(\lambda_{1}),\ldots,\rho(\lambda_{nev})$ 
are all nonzero, the following relation holds:
\begin{equation} \label{spann}
{\rm span}\{y^{(1)},\ldots,y^{(nev)}\} \subseteq {\rm range} \left(\Re e \left \lbrace \sum_{\ell=1}^{N_c} \omega_{\ell}S_{\zeta_{\ell}}^{-1} \right \rbrace \right).
\end{equation}
The above relation suggests to compute an approximation to $\operatorname{span}\{y^{(1)},\ldots,
y^{(nev)}\}$ by capturing the range space of $ \Re e \left \lbrace \sum_{\ell=1}^{N_c} \omega_{\ell}
S_{\zeta_{\ell}}^{-1} \right \rbrace $.

\subsection{A Krylov-based approach} \label{practical_form}

To capture ${\rm range} \left(\Re e \left \lbrace \sum_{\ell=1}^{N_c} \omega_{\ell}S_{\zeta_{\ell}}^{-1} \right \rbrace \right)$ 
we consider the numerical scheme outlined in Algorithm \ref{alg:arnoldi}. In contrast with RF-KRYLOV, Algorithm \ref{alg:arnoldi} 
is based on the Lanczos process \cite{10.2307/2007563}. Variable $T_{\mu}$ denotes a $\mu \times \mu$ symmetric tridiagonal matrix 
with $\alpha_1,\ldots,\alpha_{\mu}$ as its diagonal entries, and $\beta_1,\ldots,\beta_{\mu-1}$ as its off-diagonal entries, 
respectively. Line 2 computes the ``filtered'' vector $w$ by applying $ \Re e \left \lbrace \sum_{\ell=1}^{N_c} \omega_{\ell} 
S_{\zeta_{\ell}}^{-1}\right\rbrace$ to $q^{(\mu)}$ by solving the $N_c$ linear systems associated with matrices $S_{\zeta_{\ell}},\ 
\ell=1,\ldots,N_c$. Lines 4-12 orthonormalize $w$ against vectors $q^{(1)},\ldots,q^{(\mu)}$ in order to generate the next vector 
$q^{(\mu+1)}$. Algorithm \ref{alg:arnoldi} terminates when the trace of the tridiagonal matrices $T_{\mu}$ and $T_{\mu-1}$ remains 
the same up to a certain tolerance. 

\vspace{0.1in}
\begin{minipage}{\dimexpr\linewidth-9\fboxsep-9\fboxrule\relax}
\vbox{
\begin{algorithm0}{Krylov restricted to the interface variables} \label{alg:arnoldi}
\betab
\>0. \> Start with $q^{(1)} \in \mathbb{R}^s,\ s.t.\ \|q^{(1)}\|_2 = 1,$ $q_0 := 0,$ $b_1 = 0$, ${\rm tol}\in \mathbb{R}$                                    \\
\>1. \> For $\mu=1,2,\ldots$                                                                                                                                 \\
\>2. \> \> Compute $w= \Re e \left\lbrace \sum_{\ell=1}^{N_c}\omega_{\ell} S_{\zeta_{\ell}}^{-1}q^{(\mu)} \right\rbrace - b_\mu q^{(\mu-1)}$                 \\
\>3. \> \> $a_\mu=w^T q^{(\mu)}$                                                                                                                             \\
\>4. \> \> For $\kappa=1,\ldots,\mu$                                                          \\
\>5. \> \> \> $w = w - q^{(\kappa)}(w^T q^{(\kappa)})$                                        \\
\>6. \> \> End                                                                                \\
\>7. \> \> $b_{\mu+1} := \|w\|_2$                                                                                                            \\
\>8. \> \> If $b_{\mu+1} = 0$												                     \\
\>9. \>\>\> generate a unit-norm $q^{(\mu+1)}$ orthogonal to $q^{(1)},\ldots, q^{(\mu)}$					             \\
\>10. \>\> Else																     \\
\>11. \> \> \> $q^{(\mu+1)} = w / b_{\mu+1}$  												     \\
\>12 \> \> EndIf 															     \\
\>13. \> \> If the sum of eigenvalue of $T_{\mu}$ remains unchanged (up to ${\rm tol}$)  					             \\
\>   \> \> during the last few iterations; BREAK; EndIf                                       \\
\>14.\> End                                                                                   \\
\>15.\> Return $Q_{\mu} = [q^{(1)},\ldots, q^{(\mu)}]$                                        \\
\entab
\end{algorithm0}
}
\end{minipage}

\vspace{0.1in}

Algorithm \ref{alg:arnoldi} and RF-KRYLOV share a few key differences. First, Algorithm
\ref{alg:arnoldi} restricts orthonormalization to vectors of length $s$ instead of $n$. 
In addition, Algorithm \ref{alg:arnoldi} only requires linear system solutions with 
$S_\zeta$ instead of $A-\zeta M$. As can be verified by (\ref{inv_dd2}), a computation 
of the form $(A-\zeta M)^{-1}v=w,\ \zeta \in \mathbb{C}$ requires -in addition to a linear 
system solution with matrix $S_{\zeta}$- two linear system solutions with $B_{\zeta}$ as
well as two Matrix-Vector multiplications with $E_{\zeta}$. Finally, in contrast to RF-KRYLOV 
which requires at least $nev$ iterations to compute any $nev$ eigenpairs of the pencil 
$(A,M)$, Algorithm \ref{alg:arnoldi} might terminate in fewer than $nev$ iterations. This 
possible ``early termination'' of Algorithm \ref{alg:arnoldi} is explained in more detail 
by Proposition \ref{lowrank}.

\begin{proposition} \label{lowrank}
The rank of the matrix $\Re e \left \lbrace \sum_{\ell=1}^{N_c} \omega_{\ell} S_{\zeta_{\ell}}^{-1} 
\right \rbrace$, 
\begin{equation} r(S)={\rm rank} \left(\Re e \left \lbrace \sum_{\ell=1}^{N_c} \omega_{\ell} S_{\zeta_{\ell}}^{-1} \right \rbrace \right),\end{equation}
satisfies the inequality 
\begin{equation} \label{rS}{\rm rank} \left(\left[y^{(1)},\ldots,y^{(nev)}\right]\right) \leq r(S) \leq s.\end{equation}
\end{proposition}

\begin{proof}
We first prove the upper bound of $r(S)$. Since $ \Re e \left \lbrace \sum_{\ell=1}^{N_c} 
\omega_{\ell}S_{\zeta_{\ell}}^{-1} \right \rbrace$ is of size $s\times s$, $r(S)$ can not 
exceed $s$. To get the lower bound, let $\rho(\lambda_i)=0,\ i=nev+\kappa,\ldots,n$, 
where $0\leq \kappa \leq n-nev$. We then have 
\[\Re e \left \lbrace \sum_{\ell=1}^{N_c} \omega_{\ell} S_{\zeta_{\ell}}^{-1} \right 
\rbrace = \sum_{i=1}^{nev+\kappa} \rho(\lambda_i) y^{(i)} (y^{(i)})^T,\] and ${\rm rank} 
\left( \Re e\left \lbrace \sum_{\ell=1}^{N_c} \omega_{\ell}S_{\zeta_{\ell}}^{-1} \right 
\rbrace\right)={\rm rank}\left(\left[y^{(1)},\ldots,y^{(nev+\kappa)} \right]\right)$. 
Since $\rho(\lambda_i)\neq 0,\ i=1,\ldots,nev$, we have
\[r(S) = {\rm rank} \left(\left[y^{(1)},\ldots,y^{(nev+\kappa)}\right]\right) \geq {\rm rank}\left(\left[y^{(1)},\ldots,y^{(nev)} \right]\right).\]
\end{proof} 

\begin{figure} 
\centering
\includegraphics[width=0.79\textwidth]{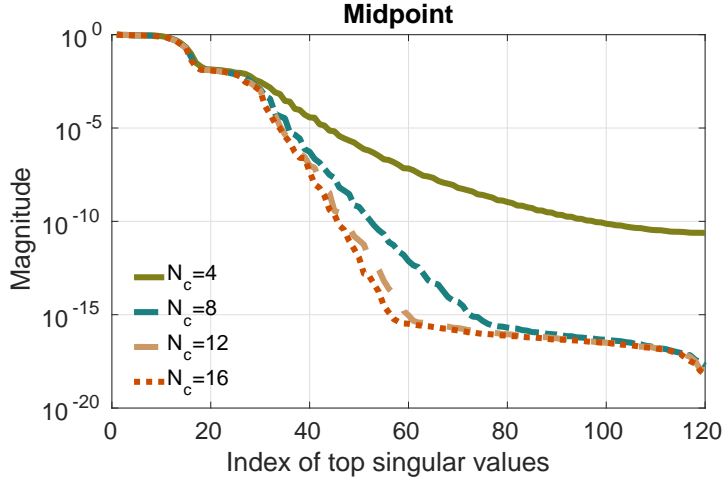}
\caption{The leading singular values of $\Re e \left \lbrace \sum_{\ell=1}^{N_c} 
\omega_{\ell} S_{\zeta_{\ell}}^{-1}\right \rbrace$ for different values of $N_c$. 
\label{fig:schur_svd}}
\end{figure}

By Proposition \ref{lowrank}, Algorithm \ref{alg:arnoldi} will perform at most $r(S)$ iterations, 
and $r(S)$ can be as small as ${\rm rank}\left(\left[y^{(1)},\ldots,y^{(nev)} \right]\right)$. 
We quantify this with a short example for a 2D Laplacian matrix generated by a Finite Difference
discretization with Dirichlet boundary conditions (for more details on this matrix see entry
``FDmesh1'' in Table \ref{testmat}) where we 
set $[\alpha,\beta]=[\lambda_1,\lambda_{100}]$ (thus $nev=100$). After computing vectors $y^{(1)},\ldots,y^{(nev)}$ 
explicitly, we found that ${\rm rank} \left(\left [y^{(1)},\ldots,y^{(nev)}\right] \right)=48$. 
Figure \ref{fig:schur_svd} plots the 120 (after normalization) leading singular values of matrix 
$\Re e \left \lbrace\sum_{\ell=1}^{N_c} \omega_{\ell} S_{\zeta_{\ell}}^{-1}\right \rbrace$. As $N_c$ 
increases, the trailing $s-{\rm rank} \left(\left [y^{(1)},\ldots,y^{(nev)}\right] \right)$ singular 
values approach zero. Moreover, even for those singular values which are not zero, their magnitude 
might be small, in which case Algorithm \ref{alg:arnoldi} might still converge in fewer than $r(S)$ 
iterations. Indeed, when $N_c=16$, Algorithm \ref{alg:arnoldi} terminates after exactly $36$ 
iterations which is lower than $r(S)$ and only one third of the minimum number of iterations 
required by RF-KRYLOV for any value of $N_c$. As a sidenote, when $N_c=2$, Algorithm \ref{alg:arnoldi} 
terminates after $70$ iterations.


\section{Approximation of $\operatorname{span}\{u^{(1)},\ldots,u^{(nev)}\}$} \label{interior_part}

Recall the partitioning of eigenvector $x^{(i)}$ in (\ref{eig_part}) and assume that its interface 
part $y^{(i)}$ is already computed. A straightforward approach to recover $u^{(i)}$ is then to solve
the linear system in $(\ref{schur_eigvec2})$. However, this entails two drawbacks. First, solving 
the linear systems with each different $B-\lambda_iM_B$ for all $\lambda_i,\ i=1,\ldots,nev$ might 
become prohibitively expensive when $nev \gg 1$. More importantly, Algorithm \ref{alg:arnoldi} only 
returns an approximation to ${\rm span}\{y^{(1)},\ldots,y^{(nev)}\}$, rather than the individual 
vectors $y^{(1)},\ldots,y^{(nev)}$, or the eigenvalues $\lambda_1,\ldots,\lambda_{nev}$. 

In this section we alternatives for the approximation of ${\rm span}\left \lbrace u^{(1)},\ldots,
u^{(nev)} \right \rbrace$. Since the following discussion applies to all $n$ eigenpairs 
of $(A,M)$, we will drop the superscripts in $u^{(i)}$ and $y^{(i)}$.

\subsection{The basic approximation}

To avoid solving the $nev$ different linear systems in $(\ref{schur_eigvec2})$ we consider the same 
real scalar $\sigma$ for all $nev$ sought eigenpairs. The part of each sought eigenvector $x$ 
corresponding to the interior variables, $u$, can then be approximated by 
\begin{equation} \hat{u} = -B_{\sigma}^{-1}E_{\sigma}y. \label{prol2} \end{equation} 
In the following proposition, we analyze the difference between $u$ and its approximation $\hat{u}$ 
obtained by (\ref{prol2}). 
\begin{lemma} \label{prop:error}
Suppose $u$ and $\hat{u}$ are computed as in $(\ref{schur_eigvec2})$ and $(\ref{prol2})$, respectively. 

Then:
\begin{equation}
\label{eq:erruhat}
u-\hat{u} = -[B_{\lambda}^{-1}-B_{\sigma}^{-1}]E_{\sigma} y + (\lambda-\sigma)B_{\lambda}^{-1} M_Ey.
\end{equation} 
\end{lemma}
\begin{proof}
We can write $u$ as
\begin{equation} \label{carousel1}
\begin{aligned}
 u & =-B_{\lambda}^{-1} E_{\lambda}y                                                 \\
   & =-B_{\lambda}^{-1}(E_{\sigma}-(\lambda-\sigma)M_E)y                             \\
   & =-B_{\lambda}^{-1}E_{\sigma}y + (\lambda-\sigma)B_{\lambda}^{-1}M_Ey.
\end{aligned}
\end{equation}
The result in (\ref{eq:erruhat}) follows by combining (\ref{prol2}) and (\ref{carousel1}).
\end{proof}

We are now ready to compute an upper bound of $u-\hat{u}$ measured in the $M_B$-norm.\footnote{We define the $X$-norm of any nonzero vector $y$ and SPD matrix $X$ as $||y||_{X}=\sqrt{y^TXy}$.}
\begin{theorem}
\label{thm:uuhat}
Let the eigendecomposition of $(B,M_B)$ be written as
\begin{equation}
BV = M_BVD,
\label{eq:eigbm}
\end{equation}
where $D = {\rm diag}(\delta_1,
\ldots,\delta_d)$ and $V=[v^{(1)},\ldots,v^{(d)}]$. If $\hat{u}$ is defined as in (\ref{prol2}) and 
$(\delta_{\ell},v^{(\ell)}),\ \ell=1,\ldots,d$ denote the eigenpairs of $(B,M_B)$ with $(v^{(\ell)})^T
M_Bv^{(\ell)}=1$, then
\begin{equation}
 \left\Vert u-\hat{u}\right\Vert_{M_B} \leq \max_{\ell} \frac{|\lambda-\sigma|}{|(\lambda-\delta_\ell)(\sigma-\delta_\ell)|} ||E_{\sigma}y||_{M^{-1}_B} + 
 \max_{\ell} \frac{|\lambda-\sigma|}{|\lambda-\delta_\ell|} ||M_{E}y||_{M^{-1}_B},
\end{equation}
\end{theorem}
\begin{proof}
Since $M_B$ is SPD, vectors $E_{\sigma}y$ and $M_Ey$ in (\ref{carousel1}) can be expanded in the basis $M_Bv^{(\ell)}$ as:
\begin{equation}
 E_{\sigma}y = M_B \sum_{\ell=1}^{\ell=d} \epsilon_{\ell} v^{(\ell)},\ \ \ M_Ey = M_B \sum_{\ell=1}^{\ell=d} \gamma_{\ell} v^{(\ell)}, \label{carousel2}
\end{equation}
where $\epsilon_{\ell},\ \gamma_{\ell} \in \mathbb{R}$ are the expansion coefficients. Based on (\ref{eq:eigbm}) 
and noting that $V^TM_BV=I$, shows that
\begin{equation}
B_{\sigma}^{-1} =V(D-\sigma I)^{-1}V^{T}, \ \ \ B_{\lambda}^{-1}=V(D-\lambda I)^{-1}V^{T}. \label{eq:bb}
\end{equation} 
Substituting $(\ref{carousel2})$ and $(\ref{eq:bb})$ into the right-hand side of (\ref{eq:erruhat}) gives
\begin{align*}
u-\hat{u} = & -V\left[(D-\lambda I)^{-1}-(D-\sigma I)^{-1}\right]V^{T}\left(M_B \sum_{\ell=1}^{\ell=d} 
\epsilon_{\ell} v^{(\ell)}\right) \\
& + (\lambda-\sigma)V(D-\lambda I)^{-1}V^{T} \left(M_B\sum_{\ell=1}^{\ell=d} \gamma_{\ell} v^{(\ell)}\right) \\
 = & -\sum_{\ell=1}^{\ell=d} \frac{\epsilon_{\ell}(\lambda-\sigma)}
 {(\delta_\ell-\lambda)(\delta_\ell-\sigma)} v^{(\ell)} + \sum_{\ell=1}^{\ell=d} \frac{\gamma_{\ell}(\lambda-\sigma)}{\delta_\ell-\lambda} v^{(\ell)}.
\end{align*}
Now, taking the $M_B$-norm of the above equation, we finally obtain
\begin{align*}
||u-\hat{u}||_{M_B} & \leq \left\Vert \sum_{\ell=1}^{\ell=d} \frac{\epsilon_{\ell}(\lambda-\sigma)}{(\delta_\ell-\lambda)(\delta_\ell-\sigma)} v^{(\ell)}\right\Vert_{M_B} 
                           + \left\Vert \sum_{\ell=1}^{\ell=d} \frac{\gamma_{\ell}(\lambda-\sigma)}{\delta_\ell-\lambda} v^{(\ell)}\right\Vert_{M_B}                 \\
                    & = \left\Vert \sum_{\ell=1}^{\ell=d} \left|\frac{(\lambda-\sigma)}{(\delta_\ell-\lambda)(\delta_\ell-\sigma)}\right| \epsilon_{\ell}v^{(\ell)}\right\Vert_{M_B} 
                           + \left\Vert \sum_{\ell=1}^{\ell=d} \left|\frac{(\lambda-\sigma)}{\delta_\ell-\lambda}\right| \gamma_{\ell}v^{(\ell)}\right\Vert_{M_B}                 \\
                    & \leq \max_{\ell} \frac{|\lambda-\sigma|}{|(\lambda-\delta_\ell)(\sigma-\delta_\ell)|} \left\Vert \sum_{\ell=1}^{\ell=d} \epsilon_\ell v^{(\ell)} \right\Vert_{M_B} 
                           + \max_{\ell} \frac{|\lambda-\sigma|}{|\lambda-\delta_\ell|} \left\Vert \sum_{\ell=1}^{\ell=d}  \gamma_\ell v^{(\ell)}\right\Vert_{M_B}   \\
 &= \max_{\ell} \frac{|\lambda-\sigma|}{|(\lambda-\delta_\ell)(\sigma-\delta_\ell)|} \left\Vert M^{-1}_{B}E_{\sigma}y\right\Vert_{M_B} + 
 \max_{\ell} \frac{|\lambda-\sigma|}{|\lambda-\delta_\ell|} \left\Vert M^{-1}_{B}M_{E}y\right\Vert_{M_B}\\
 &= \max_{\ell} \frac{|\lambda-\sigma|}{|(\lambda-\delta_\ell)(\sigma-\delta_\ell)|} \left\Vert E_{\sigma}y\right\Vert_{M^{-1}_B} + 
 \max_{\ell} \frac{|\lambda-\sigma|}{|\lambda-\delta_\ell|} \left\Vert M_{E}y\right\Vert_{M^{-1}_B}.
\end{align*}
\end{proof}

Theorem \ref{thm:uuhat} indicates that the upper bound of $||u-\hat{u}||_{M_B}$ depends on the 
distance between $\sigma$ and $\lambda$, as well as the distance of these values from the 
eigenvalues of $(B,M_B)$. This upper bound becomes relatively large when $\lambda$ is located 
far from $\sigma$, while, on the other hand, becomes small when $\lambda$ and $\sigma$ lie 
close to each other, and far from the eigenvalues of $(B,M_B)$. 

\subsection{Enhancing accuracy by resolvent expansions} \label{inter_resol}

Consider the resolvent expansion of $B_{\lambda}^{-1}$ around $\sigma$:
\begin{equation}
B_{\lambda}^{-1}=B_{\sigma}^{-1} \sum_{\theta=0}^{\infty} \left[(\lambda-\sigma)M_BB_{\sigma}^{-1}\right]^{\theta}. \label{eq:resolventB}
\end{equation}
By (\ref{eq:erruhat}), the error $u-\hat{u}$ consists of two components: $i$) $(B_{\lambda}^{-1}-B_{\sigma}^{-1})E_{\sigma} y$; 
and $ii$) $(\lambda-\sigma) B_{\lambda}^{-1} M_Ey$. An immediate improvement is then to approximate $B_{\lambda}^{-1}$ by also 
considering higher-order terms in (\ref{eq:resolventB}) instead of $B_{\sigma}^{-1}$ only. Furthermore, the same idea can be 
repeated for the second error component. Thus, we can extract $\hat{u}$ by a projection step from the following subspace
\begin{equation}
\hat{u} \in \{B^{-1}_{\sigma}E_{\sigma}y,\ldots,B^{-1}_{\sigma}\left(M_B B_{\sigma}^{-1}\right)^{\psi-1} E_{\sigma}y,B^{-1}_{\sigma}M_Ey,\ldots,B^{-1}_{\sigma}\left(M_B B_{\sigma}^{-1}\right)^{\psi-1} M_Ey\}. \label{yet1}
\end{equation}
The following theorem refines the upper bound of $\|u-\hat{u}\|_{M_B}$ when $\hat{u}$ is approximated by the subspace in (\ref{yet1})
and $\psi\geq 1$ resolvent expansion terms are retained in (\ref{eq:resolventB}).

\begin{theorem}
\label{thm:uuhat2}
 Let ${\cal U}={\rm span}\left\lbrace U_1, U_2\right\rbrace$ where
 \begin{align}
  & U_1 = \left[B_{\sigma}^{-1} E_{\sigma}y,\ldots, B_{\sigma}^{-1} \left(M_B B_{\sigma}^{-1}\right)^{\psi-1} E_{\sigma}y\right], \\
  & U_2 = \left[B_{\sigma}^{-1} M_Ey,\ldots,B_{\sigma}^{-1}\left(M_B B_{\sigma}^{-1}\right)^{\psi-1} M_Ey\right].
  \end{align}
 If $\hat{u}:= \argmin_{g \in {\cal U}} \|u-g\|_{M_B}$, and $(\delta_{\ell},v^{(\ell)}),\ \ell=1,\ldots,d$
 denote the eigenpairs of $(B,M_B)$, then:
 \begin{equation}
 \|u-\hat{u}\|_{M_B} \leq \max_{\ell} \frac{|\lambda-\sigma|^{\psi}||E_{\sigma}y||_{M^{-1}_B}}{|(\lambda-\delta_\ell)(\sigma-\delta_\ell)^{\psi}|}   + \max_{\ell} \frac{|\lambda-\sigma|^{\psi+1}|| 
M_{E}y||_{M^{-1}_B}}{|(\lambda-\delta_\ell)(\sigma-\delta_\ell)^{\psi}|}. \label{carousel9}
 \end{equation}
\end{theorem} 
\begin{proof}
 Define a vector $g:=U_1{\bf c}_1+U_2{\bf c}_2$ where 
 \begin{equation}
  {\bf c}_1 = -\left[1,\lambda-\sigma,\ldots,(\lambda-\sigma)^{\psi-1}\right]^T,\ \ \ {\bf c}_2 = \left[\lambda-\sigma,\ldots,(\lambda-\sigma)^{\psi}\right]^T.
  \label{eq:c}
 \end{equation}
 If we equate terms, the difference between $u$ and $g$ satisfies
 \begin{align} \label{lala20}
  u-g  = & - \left[B_{\lambda}^{-1}-B_{\sigma}^{-1}\sum_{\theta=0}^{\psi-1}\left[(\lambda-\sigma)M_B B_{\sigma}^{-1}\right]^{\theta}\right]E_{\sigma} y \\
               & + (\lambda-\sigma)\left[B_{\lambda}^{-1}-B_{\sigma}^{-1}\sum_{\theta=0}^{\psi-1}\left[(\lambda-\sigma)M_B B_{\sigma}^{-1}\right]^{\theta}\right] M_Ey.\nonumber
 \end{align}
Expanding $B_{\sigma}^{-1}$ and $B_{\lambda}^{-1}$ in the eigenbasis of $(B,M_B)$ gives
\begin{align}
B_{\lambda}^{-1}-B_{\sigma}^{-1}\sum_{\theta=0}^{\psi-1}\left[(\lambda-\sigma)M_B B_{\sigma}^{-1}\right]^{\theta} & =  (\lambda-\sigma)^{\psi}\left[V(D-\lambda I)^{-1}(D-\sigma I)^{-\psi}V^{T}\right],
\end{align}
and thus (\ref{lala20}) can be simplified as
\begin{align*}
u-g  =& -(\lambda-\sigma)^{\psi}V(D-\lambda I)^{-1}(D-\sigma I)^{-\psi}V^{T}E_{\sigma}y \\
& \ + (\lambda-\sigma)^{\psi+1}V(D-\lambda I)^{-1}(D-\sigma I)^{-\psi}V^{T}M_{E}y.
\end{align*}
Plugging in the expansion of $E_\sigma y$ and $M_Ey$ defined in (\ref{carousel2}) finally leads to
 \begin{equation}
  u-g  = \sum_{\ell=1}^{\ell=d} \dfrac{-\epsilon_{\ell}(\lambda-\sigma)^{\psi}}{(\delta_{\ell}-\lambda)(\delta_{\ell}-\sigma)^{\psi}} v^{(\ell)} + \sum_{\ell=1}^{\ell=d} \dfrac{\gamma_{\ell}(\lambda-\sigma)^{\psi+1}}{(\delta_{\ell}-\lambda)(\delta_{\ell}-\sigma)^{\psi}} v^{(\ell)}.
  \label{eq:ug}
 \end{equation}
Considering the $M_B$-norm gives 
 \begin{align*}\label{carousel19}
\|u-g\|_{M_B} &\leq \left\Vert \sum_{\ell=1}^{\ell=d} \dfrac{-\epsilon_{\ell}(\lambda-\sigma)^{\psi}}{(\delta_{\ell}-\lambda)(\delta_{\ell}-\sigma)^{\psi}} v^{(\ell)}\right\Vert_{M_B} + \left\Vert \sum_{\ell=1}^{\ell=d} \dfrac{\gamma_{\ell}(\lambda-\sigma)^{\psi+1}}{(\delta_{\ell}-\lambda)(\delta_{\ell}-\sigma)^{\psi}} v^{(\ell)}\right\Vert_{M_B}\\
&\leq \max_{\ell} \frac{|\lambda-\sigma|^{\psi}}{|(\lambda-\delta_\ell)(\sigma-\delta_\ell)^{\psi}|} \left\Vert \sum_{\ell=1}^{\ell=d} \epsilon_\ell v^{(\ell)} \right\Vert_{M_B}\\ 
& \quad + \max_{\ell} \frac{|\lambda-\sigma|^{\psi+1}}{|(\lambda-\delta_\ell)(\sigma-\delta_\ell)^{\psi}|} \left\Vert \sum_{\ell=1}^{\ell=d} \gamma_\ell v^{(\ell)}\right\Vert_{M_B}\\
& = \max_{\ell} \frac{|\lambda-\sigma|^{\psi}||E_{\sigma}y||_{M^{-1}_B}}{|(\lambda-\delta_\ell)(\sigma-\delta_\ell)^{\psi}|} + \max_{\ell} \frac{|\lambda-\sigma|^{\psi+1}|| 
M_{E}y||_{M^{-1}_B}}{|(\lambda-\delta_\ell)(\sigma-\delta_\ell)^{\psi}|}.
 \end{align*}
Since $\hat{u}$ is the solution of $\min_{g \in {\cal U}} \|u-g\|_{M_B}$, it follows that $\|u-\hat{u}\|_{M_B} \leq \|u-g\|_{M_B}$.
\end{proof}

A comparison of the bound in Theorem \ref{thm:uuhat2} with the bound in Theorem \ref{thm:uuhat} indicates 
that one may expect an improved approximation when $\sigma$ is close to $\lambda$. Numerical examples in 
Section \ref{numerical_experiments} will verify that this approach enhances accuracy even when $|\sigma-\lambda|$ 
is not very small.

\subsection{Enhancing accuracy by deflation} \label{inter_defla}

Both Theorem \ref{thm:uuhat} and Theorem \ref{thm:uuhat2} imply that the approximation error 
$u-\hat{u}$ might have its largest components along those eigenvector directions associated 
with the eigenvalues of $(B,M_B)$ located the closest to $\sigma$. We can remove these directions 
by augmenting the projection subspace with the corresponding eigenvectors of $(B,M_B)$.

\begin{theorem}  \label{thm:all}
Let $\delta_1,\delta_2,\ldots,\delta_{\kappa}$ be the $\kappa$ eigenvalues of $(B,M_B)$ that 
lie the closest to $\sigma$, and let $v^{(1)},v^{(2)},\ldots,v^{(\kappa)}$ denote the 
corresponding eigenvectors. Moreover, let ${\cal U}={\rm span}\left\lbrace U_1,U_2,U_3\right
\rbrace$ where
 \begin{align}
  & U_1 = \left[B_{\sigma}^{-1} E_{\sigma}y,\ldots, B_{\sigma}^{-1} \left(M_B B_{\sigma}^{-1}\right)^{\psi-1} E_{\sigma}y\right], \\
  & U_2 = \left[B_{\sigma}^{-1} M_Ey,\ldots,B_{\sigma}^{-1}\left(M_B B_{\sigma}^{-1}\right)^{\psi-1} M_Ey\right],\\ 
  & U_3 = \left[v^{(1)},v^{(2)},\ldots,v^{(\kappa)} \right].
  \end{align}
 If $\hat{u}:= \argmin_{g \in {\cal U}} \|u-g\|_{M_B}$and $(\delta_{\ell},v^{(\ell)}),\ \ell=1,\ldots,d$
 denote the eigenpairs of $(B,M_B)$, then:
 \begin{equation}
 \|u-\hat{u}\|_{M_B} \leq  \max_{\ell>\kappa} \frac{|\lambda-\sigma|^{\psi}||E_{\sigma}y||_{M^{-1}_B}}{|(\lambda-\delta_\ell)(\sigma-\delta_\ell)^{\psi}|}   + \max_{\ell>\kappa} \frac{|\lambda-\sigma|^{\psi+1}|| 
M_{E}y||_{M^{-1}_B}}{|(\lambda-\delta_\ell)(\sigma-\delta_\ell)^{\psi}|} \label{carousel7}.
 \end{equation}
\end{theorem}
\begin{proof}
 Let us define the vector $g:=U_1{\bf c}_1+U_2{\bf c}_2+U_{3}{\bf c}_3$ where
 \begin{align*}
 {\bf c}_1 = -\left[1,\lambda-\sigma,\ldots,(\lambda-\sigma)^{\psi-1}\right]^T,\ \ \ {\bf c}_2 = \left[\lambda-\sigma,\ldots,(\lambda-\sigma)^{\psi}\right]^T, \\
 {\bf c}_3 = \left[\dfrac{\gamma_1(\lambda-\sigma)^{\psi+1}-\epsilon_{1}(\lambda-\sigma)^{\psi}}{(\delta_{1}-\lambda)(\delta_{1}-\sigma)^{\psi}},\ldots,\dfrac{\gamma_\kappa(\lambda-\sigma)^{\psi+1}-\epsilon_{\kappa}(\lambda-\sigma)^{\psi}}{(\delta_{\kappa}-\lambda)(\delta_{\kappa}-\sigma)^{\psi}}\right]^T.
 \end{align*}
 Since ${\bf c}_1$ and ${\bf c}_2$ are identical to those defined in (\ref{eq:c}), we can proceed 
 as in (\ref{eq:ug}) and subtract $U_{3}{\bf c}_3$. Then,
 \begin{align*}
  u-g = & \sum_{\ell=1}^{\ell=d} \dfrac{-\epsilon_{\ell}(\lambda-\sigma)^{\psi}}{(\delta_{\ell}-\lambda)(\delta_{\ell}-\sigma)^{\psi}} v^{(\ell)} + \sum_{\ell=1}^{\ell=d} \dfrac{\gamma_{\ell}(\lambda-\sigma)^{\psi+1}}{(\delta_{\ell}-\lambda)(\delta_{\ell}-\sigma)^{\psi}} v^{(\ell)} \\
        & -\sum_{\ell=1}^{\ell=\kappa} \dfrac{-\epsilon_{\ell}(\lambda-\sigma)^{\psi}}{(\delta_{\ell}-\lambda)(\delta_{\ell}-\sigma)^{\psi}} v^{(\ell)} - \sum_{\ell=1}^{\ell=\kappa} \dfrac{\gamma_{\ell}(\lambda-\sigma)^{\psi+1}}{(\delta_{\ell}-\lambda)(\delta_{\ell}-\sigma)^{\psi}} v^{(\ell)} \\
      = & \sum_{\ell=\kappa+1}^{\ell=d} \dfrac{-\epsilon_{\ell}(\lambda-\sigma)^{\psi}}{(\delta_{\ell}-\lambda)(\delta_{\ell}-\sigma)^{\psi}} v^{(\ell)} + \sum_{\ell=\kappa + 1}^{\ell=d} \dfrac{\gamma_{\ell}(\lambda-\sigma)^{\psi+1}}{(\delta_{\ell}-\lambda)(\delta_{\ell}-\sigma)^{\psi}} v^{(\ell)} \\  
 \end{align*}

 Considering the $M_B$-norm of $u-g$ gives 
 \begin{align*}\label{carousel19}
 \left\Vert u-\hat{g}\right\Vert_{M_B} &\leq \left\Vert \sum_{\ell=\kappa+1}^{\ell=d} \dfrac{-\epsilon_{\ell}(\lambda-\sigma)^{\psi}}{(\delta_{\ell}-\lambda)(\delta_{\ell}-\sigma)^{\psi}} v^{(\ell)}\right\Vert_{M_B} + \left\Vert \sum_{\ell=\kappa+1}^{\ell=d} \dfrac{\gamma_{\ell}(\lambda-\sigma)^{\psi+1}}{(\delta_{\ell}-\lambda)(\delta_{\ell}-\sigma)^{\psi}} v^{(\ell)}\right\Vert_{M_B}\\
 & \leq  \max_{\ell>\kappa} \frac{|\lambda-\sigma|^{\psi}||E_{\sigma}y||_{M^{-1}_B}}{|(\lambda-\delta_\ell)(\sigma-\delta_\ell)^{\psi}|}   + \max_{\ell>\kappa} \frac{|\lambda-\sigma|^{\psi+1}|| 
 M_{E}y||_{M^{-1}_B}}{|(\lambda-\delta_\ell)(\sigma-\delta_\ell)^{\psi}|},
 \end{align*}
 where, as previously, we made use of the expression of $E_\sigma y$ and $M_Ey$ in (\ref{carousel2}).
 Since $\hat{u}$ is the solution of $\min_{g \in {\cal U}} \|u-g\|_{M_B}$, it follows that $\|u-\hat{u}\|_{M_B} \leq \|u-\hat{g}\|_{M_B}$.
\end{proof}


\section{The RF-DDES algorithm}   \label{main_algorithm}

In this section we describe RF-DDES in terms of a formal algorithm. 

RF-DDES starts by calling a graph partitioner to partition the graph of $|A|+|M|$ into 
$p$ subdomains and reorders the matrix pencil $(A,M)$ as in (\ref{dd1}). RF-DDES then 
proceeds to the computation of those eigenvectors associated with the $nev_B^{(j)}$ 
smallest (in magnitude) eigenvalues of each matrix pencil $(B_j-\sigma M_B^{(j)},M_B^{(j)})$, 
and stores these eigenvectors in $V_j\in \mathbb{R}^{d_j \times nev_B^{(j)}},\ j=1,\ldots,p$. 
As our current implementation stands, these eigenvectors are computed by Lanczos combined 
with shift-and-invert; see \cite{Grimes:1994:SBL:190505.190526}. Moreover, while in this 
paper we do not consider any special mechanisms to set the value of $nev_B^{(j)}$, it is 
possible to adapt the work in \cite{doi:10.1137/040613767}. The next step of RF-DDES is 
to call Algorithm \ref{alg:arnoldi} and approximate ${\rm span}\{y^{(1)},\ldots,y^{(nev)}\}$ 
by ${\rm range}\{Q\}$, where $Q$ denotes the orthonormal matrix returned by Algorithm 
\ref{alg:arnoldi}. RF-DDES then builds an approximation subspace as described in Section 
\ref{RRP} and performs a Rayleigh-Ritz (RR) projection to extract approximate eigenpairs of 
$(A,M)$. The complete procedure is shown in Algorithm \ref{alg:main}.

\vspace{0.1in}
\begin{minipage}{\dimexpr\linewidth-9\fboxsep-9\fboxrule\relax}
\vbox{
\begin{algorithm0}{RF-DDES} \label{alg:main}
\betab
\>0. \> {Input:} $A,\ M,\ \alpha,\ \beta,\ \sigma,\ p,\ \{\omega_{\ell},\zeta_{\ell}\}_{\ell=1,\ldots,N_c},\ \{nev_B^{(j)}\}_{j=1,\ldots,p},\ \psi$                                  \\
\>1. \> Reorder $A$ and $M$ as in (\ref{dd1})                                                                                                                                        \\
\>2. \> For $j=1,\ldots,p$: 														                                             \\
\>3. \> \> Compute the eigenvectors associated with the $nev_B^{(j)}$ smallest \\ \> \> \> (in magnitude) eigenvalues of $(B^{(j)}_{\sigma}, M_B^{(j)})$ and store them in $V_j$     \\
\>4. \> End																	                                     \\
\>5. \> Compute $Q$ by Algorithm\_\ref{alg:arnoldi}                                                                                                                          \\
\>6. \> Form $Z$ as in (\ref{eq:z})                                                                                                                                                  \\
\>7. \> Solve the Rayleigh-Ritz eigenvalue problem: $Z^TAZ G = Z^TMZ G\hat{\Lambda}$                                                                                                 \\
\>8. \> If eigenvectors were also sought, permute the entries of each \\
\>   \> approximate eigenvector back to their original ordering    \\
\entab
\end{algorithm0}
}
\end{minipage}

The Rayleigh-Ritz eigenvalue problem at step 7) of RF-DDES can be solved either 
by a shift-and-invert procedure or by the appropriate routine in LAPACK \cite{Dongarra:1997:SUG:265932}.

\subsection{The projection matrix $Z$} \label{RRP}

Let matrix $Q$ returned by Algorithm \ref{alg:arnoldi} be written in its distributed form 
among the $p$ subdomains,
\begin{equation}
Q = 
\begin{pmatrix}
 Q_1     \\[0.3em]
 Q_2     \\[0.3em]
 \vdots  \\[0.3em]
 Q_p     \\[0.3em]
\end{pmatrix},\label{eq:q}
\end{equation}
where $Q_j\in \mathbb{R}^{s_i \times \mu},\ j=1,\ldots,p$ is local to the $j$th subdmonain
and $\mu\in \mathbb{N}^*$ denotes the total number of iterations performed by Algorithm 
\ref{alg:arnoldi}. By defining
\begin{equation} \label{trabala}
\begin{aligned}
& B^{(j)}_{\sigma}  = B_j-\sigma M_B^{(j)},  \\
&\Phi^{(j)}_{\sigma}= \left(E_j-\sigma M_E^{(j)}\right) Q_j, \\ 
&\Psi^{(j)}         = M_E^{(j)} Q_j,
\end{aligned}
\end{equation}
the Rayleigh-Ritz projection matrix $Z$ in RF-DDES can be written as:
\begin{equation}
Z = 
\begin{pmatrix}
 V_1     &        &       &        & -\Sigma_1^{(\psi)}    & \Gamma_1^{(\psi)}  \\[0.3em]
         &  V_2   &       &        & -\Sigma_2^{(\psi)}    & \Gamma_2^{(\psi)}  \\[0.3em]
         &        &\ddots &        & \vdots                & \vdots             \\[0.3em]
         &        &       &   V_p  & -\Sigma_p^{(\psi)}    & \Gamma_p^{(\psi)}  \\[0.3em]
         &        &       &        &           [Q,0_{s, (\psi-1)\mu}]            &                 \\[0.3em]
\end{pmatrix},\label{eq:z}
\end{equation}
where $0_{\chi,\psi}$ denotes a zero matrix of size $\chi \times \psi$, and
\small
\begin{equation} \label{trabala2}
\begin{aligned}
& \Sigma_j^{(\psi)}=\left[(B_{\sigma}^{(j)})^{-1}\Phi^{(j)}_{\sigma},(B_{\sigma}^{(j)})^{-1}M_B^{(j)}(B_{\sigma}^{(j)})^{-1} \Phi^{(j)}_{\sigma},\ldots,(B_{\sigma}^{(j)})^{-1}\left(M_B^{(j)}(B_{\sigma}^{(j)})^{-1}\right)^{\psi-1} \Phi^{(j)}_{\sigma}\right],\\
& \Gamma_j^{(\psi)}=\left[(B_{\sigma}^{(j)})^{-1}\Psi^{(j)},(B_{\sigma}^{(j)})^{-1}M_B^{(j)}(B_{\sigma}^{(j)})^{-1} \Psi^{(j)},\ldots,(B_{\sigma}^{(j)})^{-1}\left(M_B^{(j)}(B_{\sigma}^{(j)})^{-1}\right)^{\psi-1}\Psi^{(j)}\right].
\end{aligned}
\end{equation}
\normalsize
When $M_E$ is a nonzero matrix, the size of matrix $Z$ is $n \times (\kappa + 2 \psi \mu)$. 
However, when $M_E \equiv 0_{d,s}$, as is the case for example when $M$ is the identity 
matrix, the size of $Z$ reduces to $n \times (\kappa + \psi \mu)$ since $\Gamma_j^{(\psi)} 
\equiv 0_{d_j,\psi \mu}$. The total memory overhead associated with the $j$th subdomain 
in RF-DDES is at most that of storing $d_j(nev_B^{(j)}+2\psi \mu)+s_i\mu$ floating-point numbers.

\subsection{Main differences with AMLS}
Both  RF-DDES  and AMLS  exploit  the  domain decomposition  framework
discussed in  Section \ref{ddf}. However,  the two methods have  a few
important differences.

In  contrast to  RF-DDES which  exploits Algorithm  \ref{alg:arnoldi},
AMLS  approximates  the  part  of the  solution  associated  with  the
interface  variables of  $(A,M)$ by  solving a  generalized eigenvalue
problem  stemming  by a  first-order  approximation  of the  nonlinear
eigenvalue problem  in (\ref{schur_eigvec3}). More  specifically, AMLS
approximates            ${\rm            span}\left            \lbrace
  \left[y^{(1)},\ldots,y^{(nev)}\right]\right  \rbrace$  by  the  {\rm
  span} of the  eigenvectors associated with a few  of the eigenvalues
of  smallest magnitude  of the  SPD pencil  $(S(\sigma),-S'(\sigma))$,
where  $\sigma$  is  some  real shift  and  $S'(\sigma)$  denotes  the
derivative of  $S(.)$ at  $\sigma$.  In the  standard AMLS  method the
shift  $\sigma $  is  zero.   While AMLS  avoids  the  use of  complex
arithmetic,     a     large      number     of     eigenvectors     of
$(S(\sigma),-S'(\sigma))$ might  need be computed. Moreover,  only the
{\rm  span} of  those  vectors $y^{(i)}$  for  which $\lambda_i$  lies
sufficiently close  to $\sigma$  can be  captured very  accurately. In
contrast,      RF-DDES     can      capture      all     of      ${\rm
  span}\left            \lbrace\left[y^{(1)},\ldots,y^{(nev)}\right]\right            \rbrace$  to   high  accuracy
regardless  of where  $\lambda_i$ is  located inside  the interval  of
interest.

Another difference between RF-DDES and AMLS concerns the way in which 
 the two schemes approximate 
${\rm span}\left \lbrace\left[u^{(1)},\ldots,u^{(nev)}\right]\right \rbrace$. As can 
be easily verified, AMLS is similar to RF-DDES with the choice $\psi=1$ \cite{AMLS2}. 
While it is possible to combine AMLS with higher values of $\psi$, this might not 
always lead to a significant increase in the accuracy of the approximate eigenpairs 
of $(A,M)$ due to the inaccuracies in the approximation of ${\rm span}\left \lbrace
\left[y^{(1)},\ldots,y^{(nev)}\right]\right \rbrace$. In contrast, because RF-DDES 
can compute a good approximation to the entire space 
${\rm span}\left \lbrace\left[y^{(1)},\ldots,y^{(nev)}\right]\right \rbrace$,
the accuracy of the approximate eigenpairs of $(A,M)$ can be improved
by simply increasing $\psi$ and/or $nev_B^{(j)}$ and repeating 
the Rayleigh-Ritz projection.

\section{Experiments} \label{numerical_experiments}

In this section we present numerical experiments performed in serial and distributed 
memory computing environments. The RF-KRYLOV and RF-DDES schemes were written in 
C/C++ and built on top of the PETSc \cite{petsc-web-page,petsc-user-ref,petsc-efficient} 
and Intel Math Kernel (MKL) scientific libraries. The source files were compiled with 
the Intel MPI compiler {\tt mpiicpc}, using the -O3 optimization level. For RF-DDES, 
the computational domain was partitioned to $p$ non-overlapping subdomains by the 
METIS graph partitioner \cite{METIS-SIAM}, and each subdomain was then assigned to a 
distinct processor group. Communication among different processor groups was achieved 
by means of the Message Passing Interface standard (MPI) \cite{Snir:1998:MCR:552013}. 
The linear system solutions with matrices $A-\zeta_1M,\ldots,A-\zeta_{N_c} M$ and 
$S(\zeta_1),\ldots,S(\zeta_{N_c})$ were performed by the Multifrontal Massively Parallel 
Sparse Direct Solver (MUMPS) \cite{mumps}, while those with the block-diagonal matrices 
$B_{\zeta_1},\ldots,B_{\zeta_{N_c}}$, and $B_{\sigma}$ by MKL PARDISO \cite{intel-alt}. 

The quadrature node-weight pairs $\{\omega_{\ell},\zeta_{\ell}\},\ \ell=1,\ldots,N_c$ were 
computed by the Midpoint quadrature rule of order $2N_c$, retaining only the $N_c$ quadrature 
nodes (and associated weights) with positive imaginary part. Unless stated otherwise, 
the default values used throughout the experiments are $p=2$, $N_c=2$, and $\sigma=0$,  
while $nev_B^{(1)}=\ldots=nev_B^{(p)}=nev_B$. The stopping criterion in Algorithm \ref{alg:arnoldi}, 
was set to ${\rm tol}=1e$-$6$. All computations were carried out in 64-bit (double) precision, 
and all wall-clock times reported throughout the rest of this section will be listed in seconds.

\subsection{Computational system}

The experiments were performed on the {\tt Mesabi} Linux cluster at Minnesota Supercomputing 
Institute. {\tt Mesabi} consists of 741 nodes of various configurations with a total of 
17,784 compute cores that are part of Intel Haswell E5-2680v3 processors. Each node 
features two sockets, each socket with twelve physical cores at 2.5 GHz. Moreover, each node 
is equipped with 64 GB of system memory.

\subsection{Numerical illustration of RF-DDES} \label{model}

\begin{table}
\centering
\caption{$n$: size of $A$ and $M$, $nnz(X)$: number of non-zero entries in matrix $X$. \label{testmat}}
\begin{tabular}{l|lrrrrr}
\#&Mat. pencil        & $n$   & $nnz(A)/n$ & $nnz(M)/n$ & $[\alpha,\beta]$  & $nev$   \\
\hline
1.&bcsst24       & 3,562 & 44.89  & 1.00  & [0, 352.55]  & 100            \\
2.&Kuu/Muu       & 7,102 & 47.90  & 23.95 & [0, 934.30]  & 100            \\
3.&FDmesh1       & 24,000& 4.97   & 1.00  & [0, 0.0568]  & 100            \\
4.&bcsst39       & 46,772& 44.05  & 1.00  & [-11.76, 3915.7]  & 100       \\
5.&qa8fk/qa8fm   & 66,127& 25.11  & 25.11 & [0, 15.530]  & 100            \\
\hline
\end{tabular}
\end{table}

We tested RF-DDES on the matrix pencils listed in Table \ref{testmat}. For each pencil, 
the interval of interest $[\alpha,\beta]$ was chosen so that $nev=100$. Matrix pencils 
1), 2), 4), and 5) can be found in the SuiteSparse matrix collection 
(https://sparse.tamu.edu/) \cite{Davis:2011:UFS:2049662.2049663}.
Matrix pencil 3) was obtained by a discretization of a differential eigenvalue 
problem associated with a membrane on the unit square with Dirichlet boundary 
conditions on all four edges using Finite Differences, and is of the standard 
form, i.e., $M=I$, where $I$ denotes the identity matrix of appropriate size.

\begin{table}
\centering
\tabcolsep=0.10cm
\caption{Maximum relative errors of the approximation of the lowest $nev=100$ 
eigenvalues returned by RF-DDES for the matrix pencils in Table \ref{testmat}.}
\label{tbl:15}
\begin{tabular}{@{}l|cc|c|c|cc|c|c|cc|c|c@{}}
\toprule
\multicolumn{1}{c}{}   && \multicolumn{3}{c}{$nev_B=50$} && \multicolumn{3}{c}{$nev_B=100$} && \multicolumn{3}{c}{$nev_B=200$}     \\  \hline
                       && $\psi=1$ & $\psi=2$ & $\psi=3$ && $\psi=1$ & $\psi=2$ & $\psi=3$  && $\psi=1$ & $\psi=2$ & $\psi=3$      \\  \hline
 bcsst24               && 2.2e-2 & 1.8e-3 & 3.7e-5 && 9.2e-3 & 1.5e-5 & 1.4e-7 && 7.2e-4 & 2.1e-8 & 4.1e-11  \\  \hline
 Kuu/Muu               && 2.4e-2 & 5.8e-3 & 7.5e-4 && 5.5e-3 & 6.6e-5 & 1.5e-6 && 1.7e-3 & 2.0e-6 & 2.3e-8   \\  \hline    
 FDmesh1               && 1.8e-2 & 5.8e-3 & 5.2e-3 && 6.8e-3 & 2.2e-4 & 5.5e-6 && 2.3e-3 & 1.3e-5 & 6.6e-8   \\  \hline
 bcsst39               && 2.5e-2 & 1.1e-2 & 8.6e-3 && 1.2e-2 & 7.8e-5 & 2.3e-6 && 4.7e-3 & 4.4e-6 & 5.9e-7   \\  \hline
 qa8fk/qa8fm           && 1.6e-1 & 9.0e-2 & 2.0e-2 && 7.7e-2 & 5.6e-3 & 1.4e-4 && 5.9e-2 & 4.4e-4 & 3.4e-6   \\  
\hline\bottomrule
\end{tabular}
\end{table}
Table \ref{tbl:15} lists the maximum (worst-case) relative error among all $nev$ 
approximate eigenvalues returned by RF-DDES. In agreement with the discussion in 
Section \ref{interior_part}, exploiting higher values of $\psi$ and/or $nev_B$ 
leads to enhanced accuracy. 
Figure \ref{btm2} plots the relative errors among all $nev$ approximate eigenvalues 
(not just the worst-case errors) for the largest matrix pencil reported in Table 
\ref{testmat}. Note that ``qa8fk/qa8fm'' is a positive definite pencil, i.e., all of
its eigenvalues are positive. Since $\sigma=0$, we expect the algebraically smallest
eigenvalues of $(A,M)$ to be approximated more accurately. Then, increasing the value 
of $\psi$ and/or $nev_B$ mainly improves the accuracy of the approximation of those 
eigenvalues $\lambda$ located farther away from $\sigma$. A similar pattern was also 
observed for the rest of the matrix pencils listed in Table \ref{testmat}.


\begin{figure} 
\centering
{\includegraphics[width=0.33\textwidth]{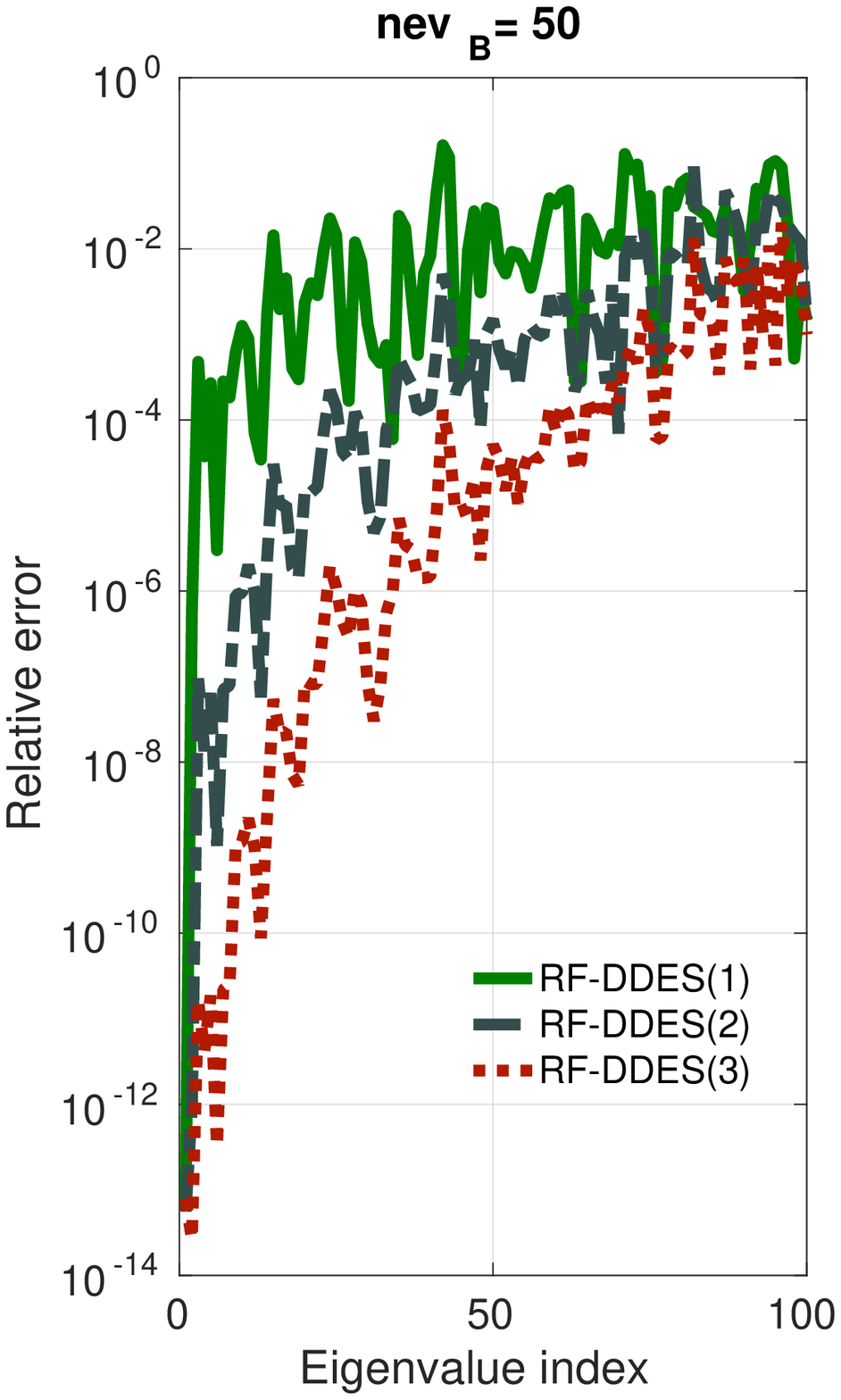}}
{\includegraphics[width=0.33\textwidth]{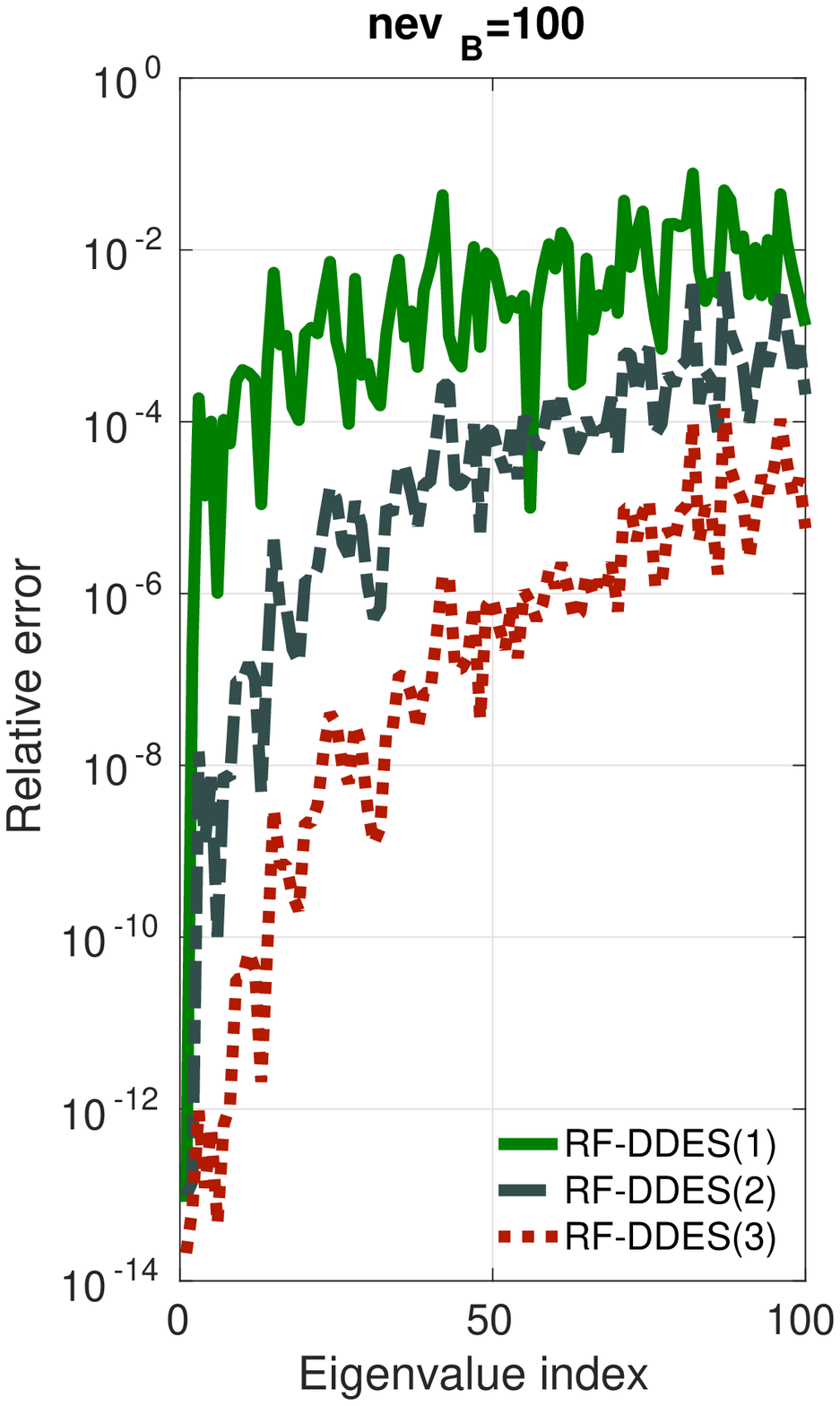}}
{\includegraphics[width=0.33\textwidth]{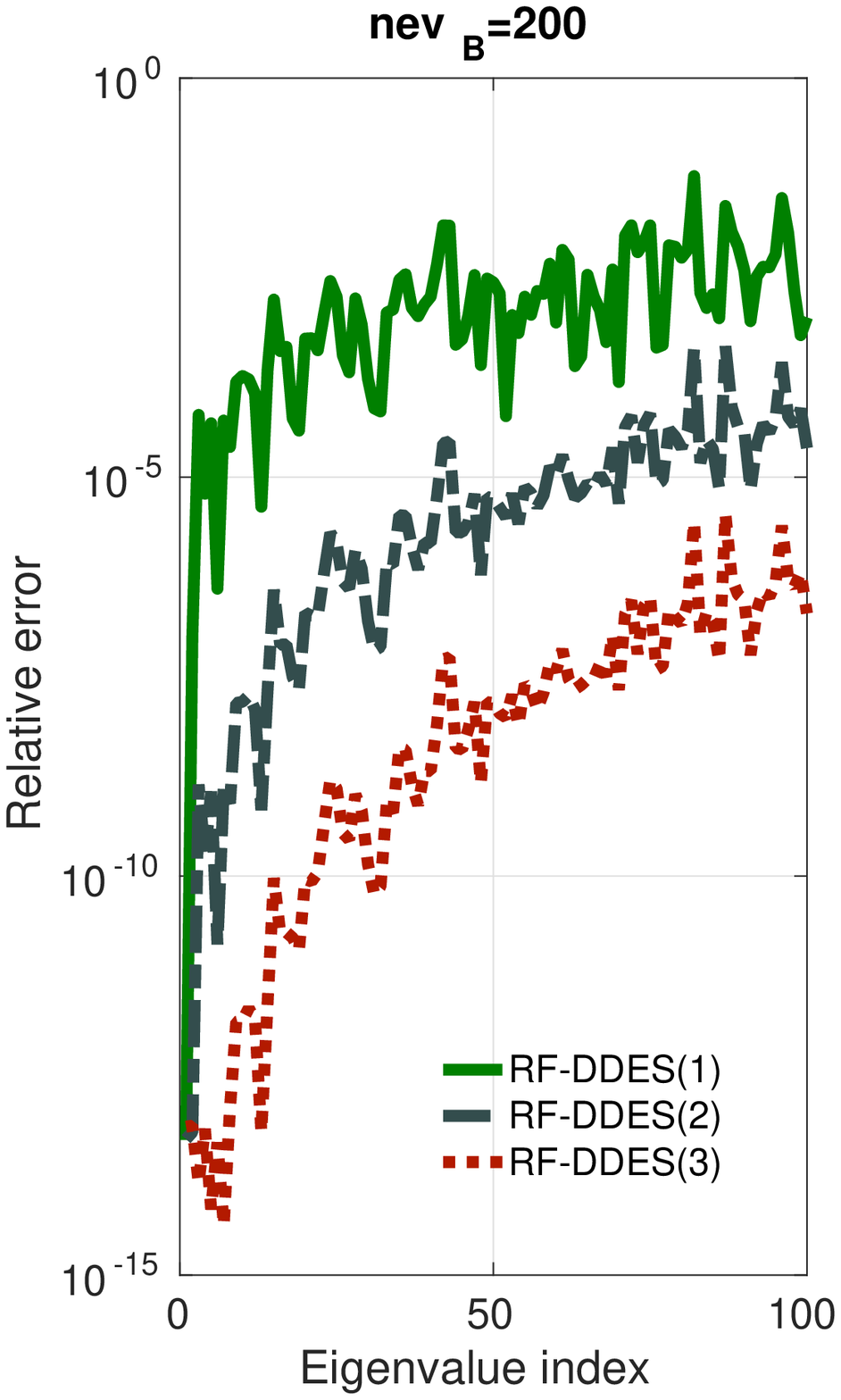}}
\caption{Relative errors of the approximation of the lowest $nev=100$ eigenvalues for the ``qa8fk/qafm'' matrix pencil.
Left: $nev_B=50$. Center: $nev_B=100$. Right: $nev_B=200$. \label{btm2}}
\end{figure}

\begin{figure} 
\centering
{\includegraphics[width=0.49\textwidth]{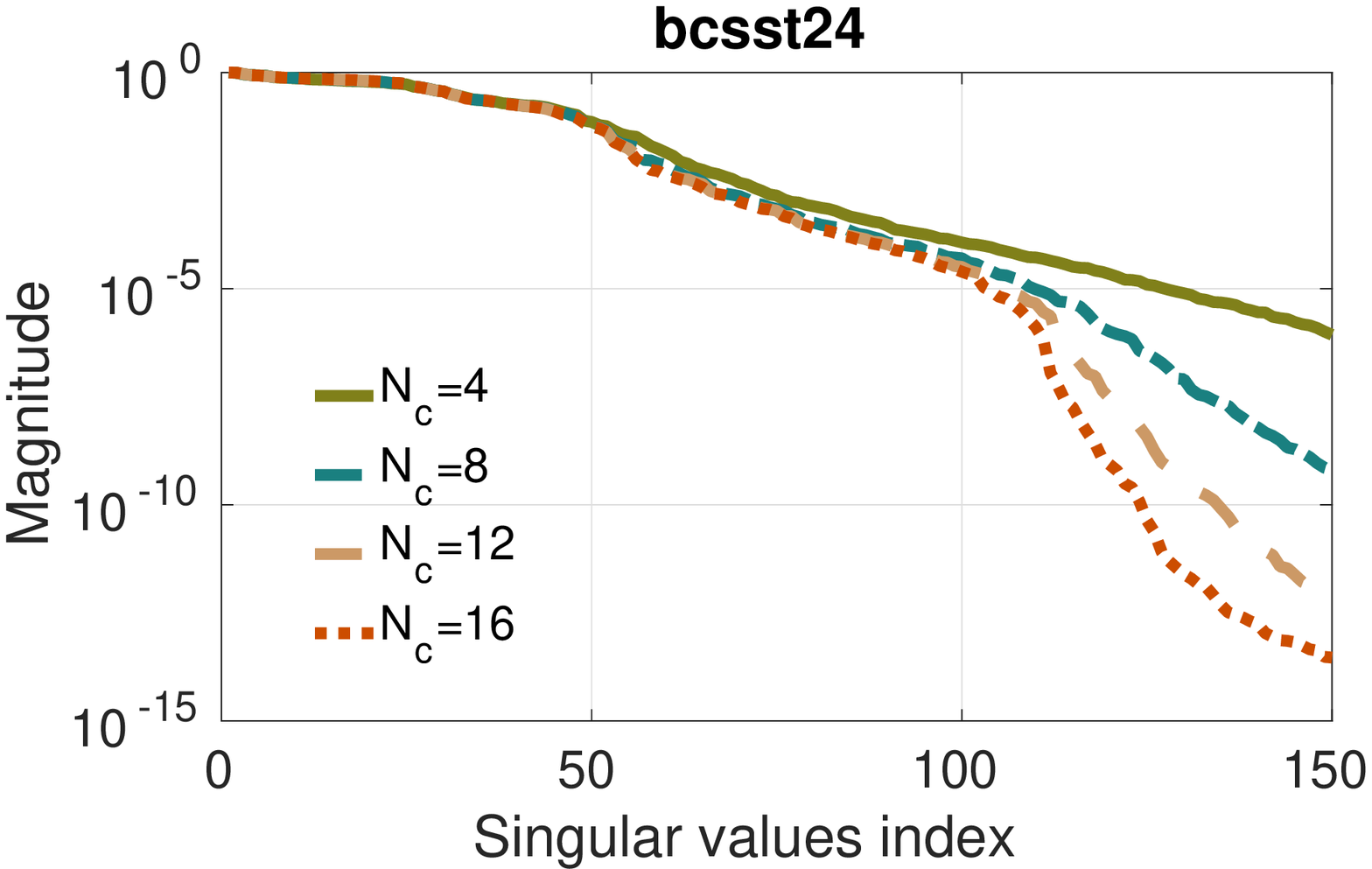}}
{\includegraphics[width=0.46\textwidth]{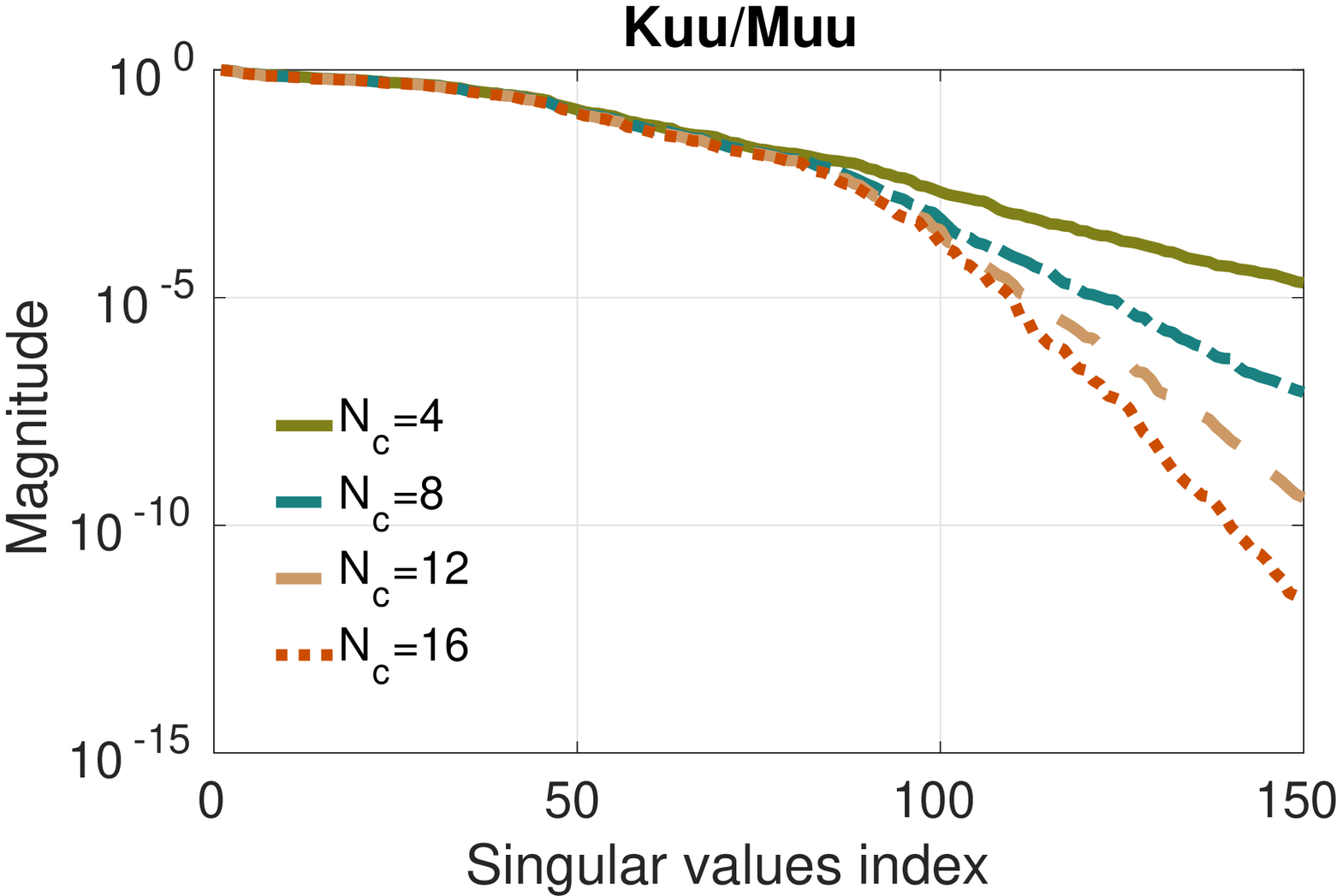}}
\caption{The 150 leading singular values of $\Re e \left \lbrace\sum_{\ell=1}^{N_c} \omega_{\ell}
S_{\zeta_{\ell}}^{-1}\right \rbrace$ for matrix pencils ``bcsst24'' and ``Kuu/Muu''. \label{fig:batch1}}
\end{figure}

\begin{table}
\centering
\caption{Number of iterations performed by Algorithm \ref{alg:arnoldi} for the 
matrix pencils listed in Table \ref{testmat}. '$s$' denotes the number of 
interface variables. \label{schur_arnoldi_iters}}
\begin{tabular}{|l|c|c|c|c|c|c|c|}
\hline
Mat. pencil & $s$  & $s/n$ & $N_c=2$ & $N_c=4$ & $N_c=8$ & $N_c=12$ & $N_c=16$ \\
\hline 
bcsst24     & 449 & 0.12 & 164 & 133 & 111 & 106 & 104                       \\ \hline
Kuu/Muu     & 720 & 0.10 & 116 & 74  & 66  & 66  & 66                        \\ \hline
FDmesh1     & 300 & 0.01 & 58  & 40  & 36  & 35  & 34                        \\ \hline 
bcsst39     & 475 & 0.01 & 139 & 93  & 75  & 73  & 72                        \\ \hline
qa8fk/qa8fm & 1272& 0.01 & 221 & 132 & 89  & 86  & 86                        \\ 
\hline
\end{tabular}
\end{table}

Table \ref{schur_arnoldi_iters} lists the number of iterations performed by Algorithm 
\ref{alg:arnoldi} as the value of $N_c$ increases. Observe that for matrix pencils 2), 
3), 4) and 5) this number can be less than $nev$ (recall the ``early termination'' 
property discussed in Proposition \ref{lowrank}), even for values of $N_c$ as low as 
$N_c=2$. Moreover, Figure \ref{fig:batch1} plots the 150 leading\footnote{After normalization by 
the spectral norm} singular values of matrix $\Re e \left \lbrace \sum_{\ell=1}^{N_c} 
\omega_{\ell}S(\zeta_{\ell})^{-1}\right \rbrace$ for matrix pencils ``bcsst24'' 
and ``Kuu/Muu'' as $N_c=4,\ 8,\ 12$ and $N_c=16$. In agreement with the discussion in 
Section \ref{practical_form}, the magnitude of the trailing $s-{\rm rank} \left(\left[y^{(1)},\ldots,y^{(nev)}
\right]\right)$ singular values approaches zero as the value of $N_c$ increases. 

Except the value of $N_c$, the number of subdomains $p$ might also affect the number 
of iterations performed by Algorithm \ref{alg:arnoldi}. Figure \ref{fig:numiters} 
shows the total number of iterations performed by Algorithm \ref{alg:arnoldi} when 
applied to matrix ``FDmesh1'' for $p=2,\ 4,\ 8$ and $p=16$ subdomains. For each 
different value of $p$ we considered $N_c=2,\ 4,\ 8,\ 12$, and $N_c=16$ quadrature 
nodes. The interval $[\alpha,\beta]$ was set so that it included only eigenvalues 
$\lambda_1,\ldots,\lambda_{200}$ ($nev=200$). Observe that higher values of $p$ might
lead to an increase in the number of iterations performed by Algorithm \ref{alg:arnoldi}.
For example, when the number of subdomains is set to $p=2$ or $p=4$, setting $N_c=2$ 
is sufficient for Algorithm \ref{alg:arnoldi} to terminate in less than $nev$ iterations 
On the other hand, when $p\geq 8$, we need at least $N_c \geq 4$ if a similar number 
of iterations is to be performed.
\begin{figure}
\centering
\begin{tikzpicture}[every plot/.append style={very thick}]
\begin{axis}[ xlabel={\# of subdomains ($p$)}, ylabel={\# of iterations}, legend style={at={(1.0051,0.7975)},anchor=west}, width=.65\textwidth, ] 
\addplot coordinates{ (2, 80) (4, 165) (8, 234) (16, 318) }; 
\addplot coordinates{ (2, 56) (4, 126) (8, 172) (16, 229) }; 
\addplot coordinates{ (2, 50) (4, 113) (8, 149) (16, 207) }; 
\addplot coordinates{ (2, 49) (4, 113) (8, 150) (16, 198) }; 
\addplot coordinates{ (2, 48) (4, 103) (8, 149) (16, 196) }; 
\legend{$N_c=2$, $N_c=4$, $N_c=8$, $N_c=12$, $N_c=16$} 
\end{axis} 
\end{tikzpicture}
\caption{Total number of iterations performed by Algorithm \ref{alg:arnoldi} when 
applied to matrix ``FDmesh1'' (where $[\alpha,\beta]=[\lambda_1,\lambda_{200}]$). 
Results reported are for all different combinations of $p=2,\ 4,\ 8$ and $p=16$, 
and $N_c=1,\ 2,\ 4,\ 8$ and $N_c=16$. \label{fig:numiters}}
\end{figure}
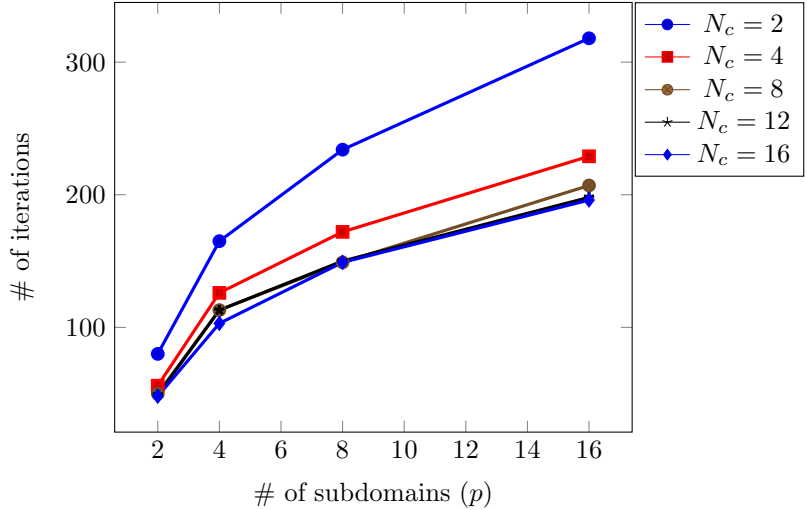
\begin{figure} 
\centering
{\includegraphics[width=0.448\textwidth]{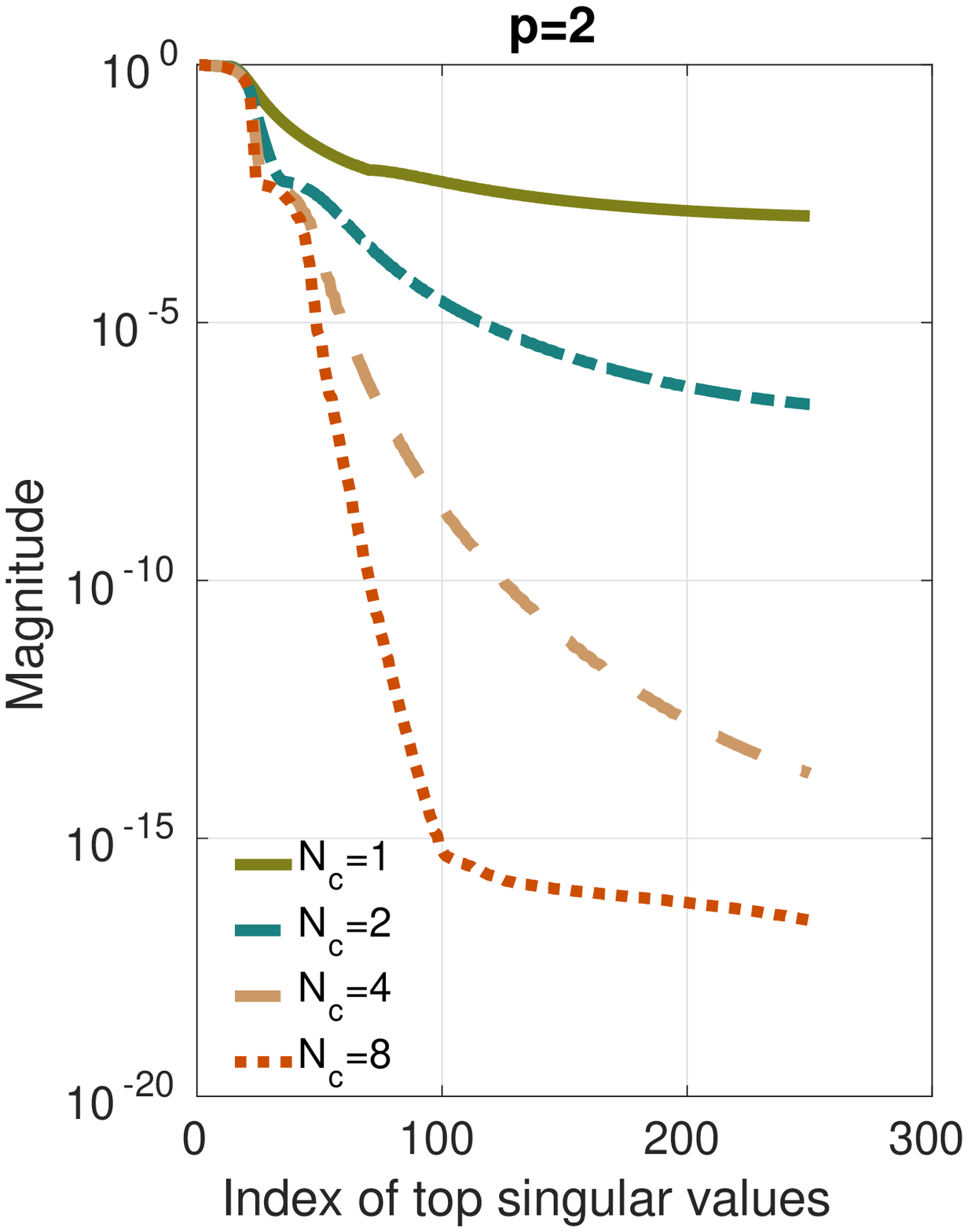}}
{\includegraphics[width=0.448\textwidth]{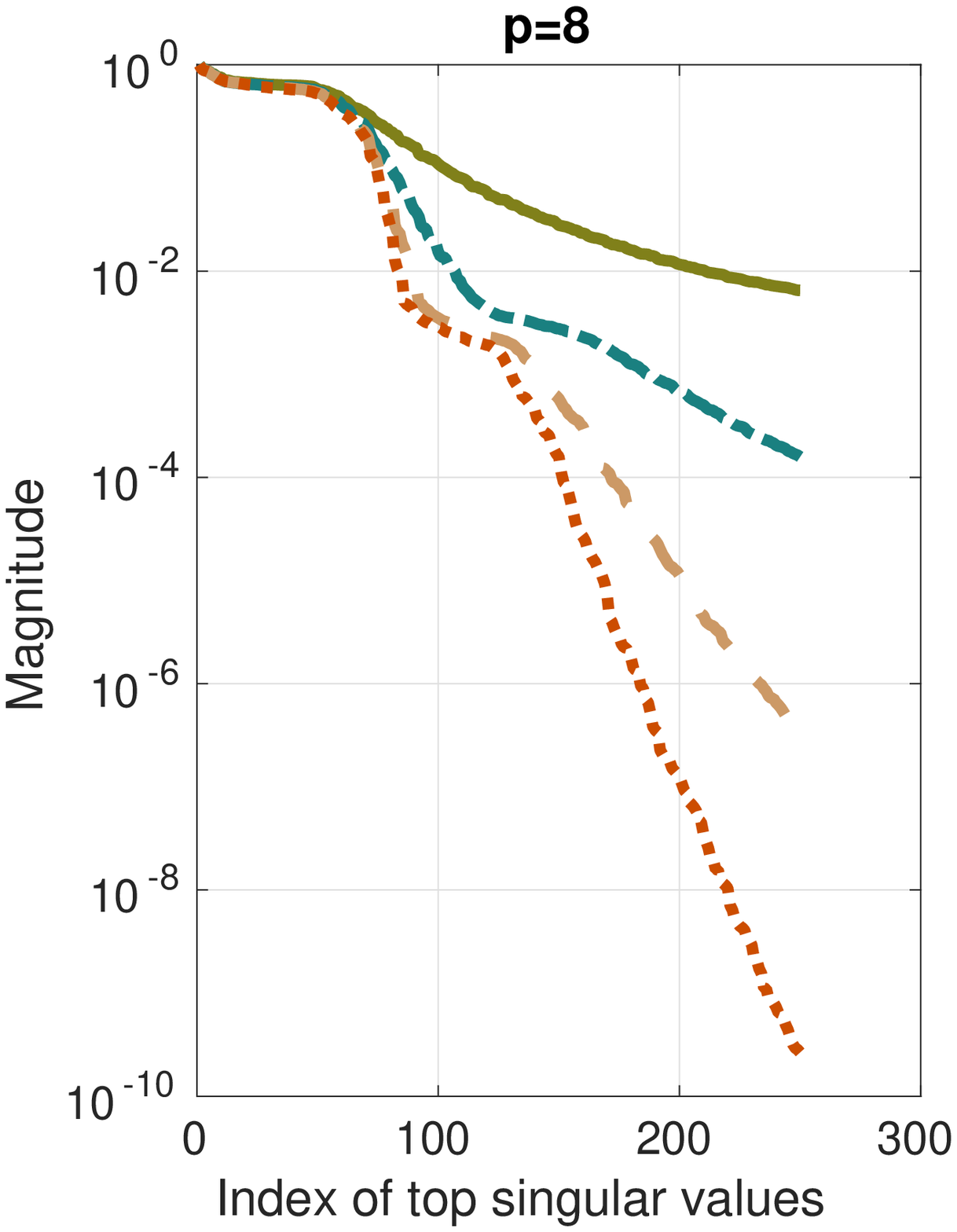}}
\caption{The leading $250$ singular values of $\Re e \left \lbrace 
\sum_{\ell=1}^{N_c} \omega_{\ell}S(\zeta_{\ell})^{-1}\right \rbrace$ 
for the same problem as in Figure \ref{fig:numiters}. Left: $p=2$. 
Right: $p=8$. For both values of $p$ we set $N_c=1,\ 2,\ 4,$ and 
$N_c=8$.\label{singvalsfdmesh1}}
\end{figure}
This potential increase in the number of iterations performed by Algorithm \ref{alg:arnoldi} 
for larger values of $p$ is a consequence of the fact that the columns of matrix 
$Y=\left[y^{(1)},\ldots,y^{(nev)}\right]$ now lie in a higher-dimensional subspace.
This might not only increase the rank of $Y$, but also affect the decay of the 
singular values of $\Re e \left \lbrace \sum_{\ell=1}^{N_c} \omega_{\ell}S(\zeta_{\ell})^{-1}\right 
\rbrace$. This can be seen more clearly in Figure \ref{singvalsfdmesh1} where we 
plot the leading $250$ singular values of 
$\Re e \left \lbrace \sum_{\ell=1}^{N_c} \omega_{\ell} S(\zeta_{\ell})^{-1}\right 
\rbrace$ of the problem in Figure \ref{fig:numiters} for two different values of 
$p$, $p=2$ and $p=8$. Note that the leading singular values decay more slowly for 
the case $p=8$. Similar results were observed for different values of $p$ and for 
all matrix pencils listed in Table \ref{testmat}.

\subsection{A comparison of RF-DDES and RF-KRYLOV in distributed computing environments}

In this section we compare the performance of  RF-KRYLOV and RF-DDES on distributed
computing environments for the matrices listed in Table \ref{testmat2}. All eigenvalue 
problems in this section are of the form $(A,I)$, i.e., standard eigenvalue problems. 
Matrices ``boneS01'' and ``shipsec8'' can be found in the SuiteSparse matrix collection.
Similarly to ``FDmesh1'', matrices ``FDmesh2'' and ``FDmesh3'' were generated by a 
Finite Differences discretization of the Laplacian operator on the unit plane using 
Dirichlet boundary conditions and two different mesh sizes so that $n=250,000$ (``FDmesh2'') 
and $n=1,000,000$ (``FDmesh3'').

\begin{table}
\small
\centering
\caption{$n$: size of $A$, $nnz(A)$: number of non-zero entries in matrix $A$. 
$s_2$ and $s_4$ denote the number of interface variables when $p=2$ and $p=4$, 
respectively. \label{testmat2}}
\begin{tabular}{l|cccccc}
\#&Matrix        & $n$       & $nnz(A)/n$ & $s_2$ & $s_4$ & $[\lambda_1,\lambda_{101},\lambda_{201},\lambda_{300}]$ \\
\hline
1.&shipsec8      & 114,919   & 28.74 & 4,534  & 9,001 & [3.2e-2, 1.14e-1, 1.57e-2, 0.20]        \\
2.&boneS01       & 172,224   & 32.03 & 10,018 & 20,451& [2.8e-3, 24.60, 45.42, 64.43]           \\
3.&FDmesh2       & 250,000   & 4.99  & 1,098  & 2,218 &[7.8e-5, 5.7e-3, 1.08e-2, 1.6e-2]        \\
4.&FDmesh3       & 1,000,000 & 4.99  & 2,196  & 4,407 & [1.97e-5, 1.4e-3, 2.7e-3, 4.0e-3]       \\
\hline
\end{tabular}
\end{table}
\normalsize

Throughout the rest of this section we will keep $N_c=2$ fixed, since this option was found 
the best both for RF-KRYLOV and RF-DDES.

\subsubsection{Wall-clock time comparisons} \label{wcexprs}
We now consider the wall-clock times achieved by RF-KRYLOV and RF-DDES when executing both schemes on 
$\tau=2,\ 4,\ 8,\ 16$ and $\tau=32$ compute cores. For RF-KRYLOV, the value of $\tau$ will denote 
the number of single-threaded MPI processes. For RF-DDES, the number of MPI processes will be equal 
to the number of subdomains, $p$, and each MPI process will utilize $\tau/p$ compute threads. Unless 
mentioned otherwise, we will assume that RF-DDES is executed with $\psi=3$ and $nev_B=100$.

\begin{table}
\centering
\tabcolsep=0.090cm
\caption{Number of iterations performed by RF-KRYLOV (denoted as RFK) and Algorithm \ref{alg:arnoldi} in RF-DDES 
(denoted by RFD(2) and RFD(4), with the number inside the parentheses denoting the value of $p$) for the matrix 
pencils listed in Table \ref{testmat2}. The convergence criterion in both RF-KRYLOV and Algorithm \ref{alg:arnoldi} 
was tested every ten iterations.} \label{tbl:16}
\resizebox{\columnwidth}{!}{
\begin{tabular}{l|cccc|cccc|cccc}
\toprule
\multicolumn{1}{c}{}   && \multicolumn{3}{c}{$nev=100$} && \multicolumn{3}{c}{$nev=200$} && \multicolumn{3}{c}{$nev=300$}   \\  \hline
            && RFK & RFD($2$) & RFD($4$) && RFK & RFD($2$) & RFD($4$) && RFK & RFD($2$) & RFD($4$)                          \\ \hline
 shipsec8   && 280 & 170 & 180 && 500 & 180 & 280 && 720 & 190 & 290 \\ \hline
 boneS01    && 240 & 350 & 410 && 480 & 520 & 600 && 620 & 640 & 740 \\ \hline
 FDmesh2    && 200 & 100 & 170 && 450 & 130 & 230 && 680 & 160 & 270 \\ \hline
 FDmesh3    && 280 & 150 & 230 && 460 & 180 & 290 && 690 & 200 & 380 \\
 \hline\bottomrule
\end{tabular}}
\end{table}

\begin{table}
\centering
\ra{0.4}
\tabcolsep=0.14cm
\caption{Maximum relative error of the approximate eigenvalues returned by RF-DDES for the matrix 
pencils listed in Table \ref{testmat2}.}
\label{tbl:17}
\resizebox{\columnwidth}{!}{
\begin{tabular}{@{}l|cc|c|c|cc|c|c|cc|c|c@{}}
\toprule
\multicolumn{1}{c}{}   && \multicolumn{3}{c}{$nev=100$} && \multicolumn{3}{c}{$nev=200$} && \multicolumn{3}{c}{$nev=300$}       \\ \hline
 $nev_B$               && $25$ & $50$ & $100$ && $25$ & $50$ & $100$  && $25$ & $50$ & $100$                \\  \hline 
 shipsec8              && 1.4e-3 & 2.2e-5 & 2.4e-6 && 3.4e-3 & 1.9e-3 & 1.3e-5 && 4.2e-3 & 1.9e-3 & 5.6e-4  \\  \hline 
 boneS01               && 5.2e-3 & 7.1e-4 & 2.2e-4 && 3.8e-3 & 5.9e-4 & 4.1e-4 && 3.4e-3 & 9.1e-4 & 5.1e-4  \\  \hline  
 FDmesh2               && 4.0e-5 & 2.5e-6 & 1.9e-7 && 3.5e-4 & 9.6e-5 & 2.6e-6 && 3.2e-4 & 2.0e-4 & 2.6e-5  \\  \hline 
 FDmesh3               && 6.2e-5 & 8.5e-6 & 4.3e-6 && 6.3e-4 & 1.1e-4 & 3.1e-5 && 9.1e-4 & 5.3e-4 & 5.3e-5  \\ 
\hline\bottomrule
\end{tabular}}
\end{table}
\begin{table}
\ra{0.7}
\tabcolsep=0.08cm
\caption{Wall-clock times of RF-KRYLOV and RF-DDES using $\tau=2,\ 4,\ 8,\ 16$ and $\tau=32$ 
computational cores. RFD(2) and RFD(4) denote RF-DDES with $p=2$ and $p=4$ subdomains, respectively.}
\label{tbl:19}
\resizebox{\columnwidth}{!}{
\begin{tabular}{@{}l|cccc|cccc|ccccc@{}}
\toprule
\multicolumn{1}{c}{}           && \multicolumn{3}{c}{$nev=100$} && \multicolumn{3}{c}{$nev=200$} && \multicolumn{3}{c}{$nev=300$}        \\ 
  Matrix                       && RFK  & RFD($2$) & RFD($4$)    && RFK & RFD($2$) & RFD($4$)     && RFK & RFD($2$) & RFD($4$)           \\ \hline
 shipsec8($\tau=2$)            && 114 & 195 & -  && 195 &207  & -   && 279 & 213  & -   \\
 \phantom{shipsec8}($\tau=4$)  && 76  & 129 & 93 && 123 &133  & 103 && 168 & 139  & 107 \\
 \phantom{shipsec8}($\tau=8$)  && 65  & 74  & 56 && 90  & 75  & 62  && 127 & 79   & 68  \\
 \phantom{shipsec8}($\tau=16$) && 40  & 51  & 36 && 66  & 55  & 41  && 92  & 57   & 45  \\
 \phantom{shipsec8}($\tau=32$) && 40  & 36  & 28 && 62  & 41  & 30  && 75  & 43   & 34  \\  \hline
  boneS01($\tau=2$)            && 94 & 292 &  -  && 194 &356  &  -   && 260&424  & -     \\
 \phantom{boneS01}($\tau=4$)   && 68 & 182 & 162 && 131 &230  & 213  && 179&277  & 260  \\
 \phantom{boneS01}($\tau=8$)   && 49 & 115 & 113 && 94  &148  & 152  && 121&180  & 187  \\
 \phantom{boneS01}($\tau=16$)  && 44 & 86  & 82  && 80  &112  & 109  && 93&137  & 132  \\
 \phantom{boneS01}($\tau=32$)  && 51 & 66  & 60  && 74  &86   &  71  && 89&105  & 79   \\  \hline
  FDmesh2($\tau=2$)            && 241& 85 &  -   && 480 &99   &  -   && 731 & 116  &  -   \\
 \phantom{FDmesh2}($\tau=4$)   && 159& 34 & 63   && 305 &37  & 78  && 473 & 43  & 85  \\
 \phantom{FDmesh2}($\tau=8$)   && 126& 22 & 23   && 228 &24  & 27  && 358 & 27  & 31  \\
 \phantom{FDmesh2}($\tau=16$)  && 89 & 16 & 15   && 171 &17  & 18  && 256 & 20  & 21  \\
 \phantom{FDmesh2}($\tau=32$)  && 51 & 12 & 12   && 94  &13  & 14  && 138 & 15  & 20  \\  \hline
  FDmesh3($\tau=2$)            && 1021& 446 & -  && 2062 & 502 & - && 3328   & 564 & -          \\
 \phantom{FDmesh3}($\tau=4$)   && 718 & 201 & 281 && 1281& 217 & 338 && 1844 & 237 & 362  \\
 \phantom{FDmesh3}($\tau=8$)   && 423 & 119 & 111 && 825 & 132 & 126 && 1250 & 143 & 141  \\
 \phantom{FDmesh3}($\tau=16$)  && 355 & 70  & 66  && 684 & 77  & 81  && 1038 & 88  & 93        \\
 \phantom{FDmesh3}($\tau=32$)  && 177 & 47  & 49  && 343 & 51  & 58  && 706  & 62  & 82        \\  
\hline\bottomrule
\end{tabular}}
\end{table}
Table \ref{tbl:19} lists the wall-clock time required by RF-KRYLOV and RF-DDES to approximate 
the $nev=100,\ 200$ and $nev=300$ algebraically smallest eigenvalues of the matrices listed 
in Table \ref{testmat2}. For RF-DDES we considered two different values of $p$; $p=2$ and $p=4$.
Overall, RF-DDES was found to be faster than RF-KRYLOV, with an increasing performance gap for 
higher values of $nev$. Table \ref{tbl:16} lists the number of iterations performed by RF-KRYLOV 
and Algorithm \ref{alg:arnoldi} in RF-DDES. For all matrices but ``boneS01'', Algorithm 
\ref{alg:arnoldi} required fewer iterations than RF-KRYLOV.  Table \ref{tbl:17} lists the 
maximum relative error of the approximate eigenvalues returned by RF-DDES when $p=4$. The 
decrease in the accuracy of RF-DDES as $nev$ increases is due the fact that $nev_B$ remains 
bounded. Typically, an increase in the value of $nev$ should be also accompanied by an increase 
in the value of $nev_B$, if the same level of maximum relative error need be retained. On the 
other hand, RF-KRYLOV always computed all $nev$ eigenpairs up to the maximum attainable accuracy.

\begin{table}
\ra{0.7}
\tabcolsep=0.08cm
\caption{Time elapsed to apply to apply the rational filter in RF-KRYLOV  and RF-DDES using
$\tau=2,\ 4,\ 8,\ 16$ and $\tau=32$ compute cores. RFD(2) and RFD(4) denote RF-DDES with 
$p=2$ and $p=4$ subdomains, respectively. For RF-KRYLOV the times listed also include the 
amount of time spent on factorizing matrices $A-\zeta_{\ell}M,\ \ell=1,\ldots,N_c$. For 
RF-DDES, the times listed also include the amount of time spent in forming and factorizing 
matrices $S_{\zeta_{\ell}},\ \ell=1,\ldots,N_c$.}
\label{tbl:20}
\resizebox{\columnwidth}{!}{
\begin{tabular}{@{}l|cccc|cccc|ccccc@{}}
\toprule
\multicolumn{1}{c}{}           && \multicolumn{3}{c}{$nev=100$}  && \multicolumn{3}{c}{$nev=200$} && \multicolumn{3}{c}{$nev=300$}        \\ 
  Matrix                       && RFK & RFD($2$) & RFD($4$)      && RFK & RFD($2$) & RFD($4$)     && RFK & RFD($2$) & RFD($4$)           \\ \hline
 shipsec8($\tau=2$)            && 104 & 153 & -&& 166 & 155 & -  && 222& 157 & -   \\
 \phantom{shipsec8}($\tau=4$)  && 71 & 93 & 75 && 107  & 96 & 80 && 137& 96 & 82   \\
 \phantom{shipsec8}($\tau=8$)  && 62 & 49 & 43 && 82  & 50 & 45  && 110& 51 & 47   \\
 \phantom{shipsec8}($\tau=16$) && 38 & 32 & 26 && 61  & 33 & 28  && 83 & 34 & 20   \\
 \phantom{shipsec8}($\tau=32$) && 39 & 21 & 19 && 59  & 23 & 20  && 68 & 24 & 22   \\ \hline
  boneS01($\tau=2$)            && 86 & 219 & - && 172 & 256 & -  && 202& 291& -    \\
 \phantom{boneS01}($\tau=4$)   && 64 & 125 & 128 && 119& 152& 168&& 150& 178& 199  \\
 \phantom{boneS01}($\tau=8$)   && 46 & 77 & 88 && 84  & 95 & 117 && 104& 112& 140  \\
 \phantom{boneS01}($\tau=16$)  && 43 & 56 & 62 && 75  & 70 & 85  && 86 & 84 & 102  \\
 \phantom{boneS01}($\tau=32$)  && 50 & 42 & 44 && 72  & 51 & 60  && 82 & 63 & 61   \\ \hline
  FDmesh2($\tau=2$)            && 227 & 52 & - && 432 & 59 & -   && 631& 65 & -    \\
 \phantom{FDmesh2}($\tau=4$)   && 152 & 22 & 36&& 287 & 24  & 42 && 426& 26 & 45   \\
 \phantom{FDmesh2}($\tau=8$)   && 122 & 13 & 14&& 215 & 14  & 16 && 335& 15 & 18   \\
 \phantom{FDmesh2}($\tau=16$)  && 85  &  9 & 8 && 164 & 10  & 10 && 242& 11 & 11   \\
 \phantom{FDmesh2}($\tau=32$)  && 50  &  6 & 6 && 90  & 7   & 8  && 127& 8  & 10   \\ \hline
  FDmesh3($\tau=2$)            && 960 & 320 & -  && 1817 & 341 & -   && 2717& 359 & -   \\
 \phantom{FDmesh3}($\tau=4$)   && 684 & 158 & 174&& 1162 & 164 & 192 && 1582& 170 & 201 \\
 \phantom{FDmesh3}($\tau=8$)   && 406 & 88  & 76&& 764  & 91 & 82  && 1114& 94  & 88    \\
 \phantom{FDmesh3}($\tau=16$)  && 347 & 45  & 43&& 656  & 48 & 49  && 976 & 51  & 52    \\
 \phantom{FDmesh3}($\tau=32$)  && 173 & 28  & 26&& 328  & 28 & 32  && 674 & 31  & 41    \\
\hline\bottomrule
\end{tabular}}
\end{table}
Table \ref{tbl:20} lists the amount of time spent on the triangular substitutions required to 
apply the rational filter in RF-KRYLOV, as well as the amount of time spent on forming and 
factorizing the Schur complement matrices and applying the rational filter in RF-DDES. For the
values of $nev$ tested in this section, these procedures were found to be the computationally most 
expensive ones.
\begin{figure} 
\centering
\subfigure[FDmesh2 ($n=250,000$).]  {\includegraphics[width=0.487\textwidth]{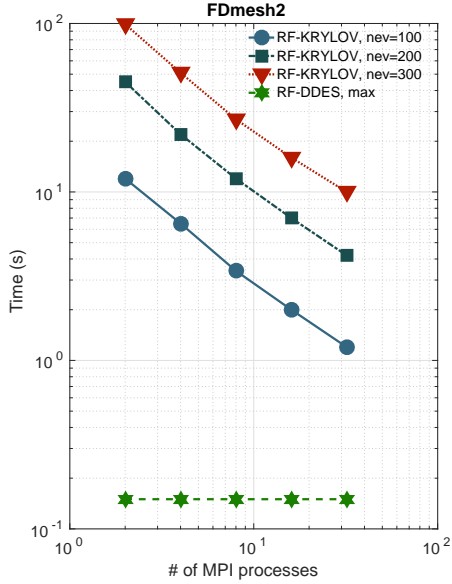}}
\subfigure[FDmesh3 ($n=1,000,000$).]{\includegraphics[width=0.49\textwidth]{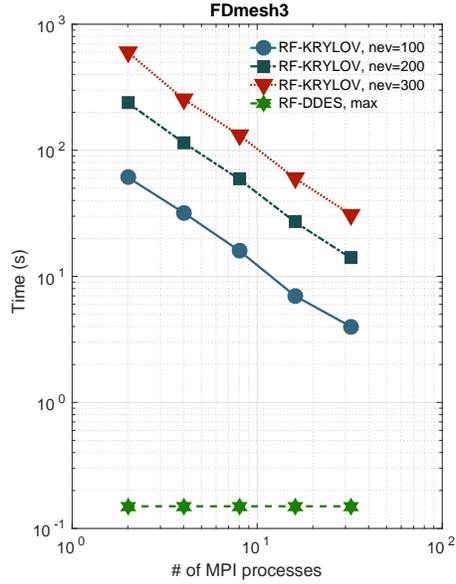}}
\caption{Time spent on orthonormalization in RF-KRYLOV and RF-DDES when computing the $nev=100,\ 
200$ and $nev=300$ algebraically smallest eigenvalues and associated eigenvectors of matrices ``FDmesh2'' 
and ``FDmesh3''. \label{fig:time_orth}}
\end{figure}
Figure \ref{fig:time_orth} plots the total amount of time spent on orthonormalization by RF-KRYLOV 
and RF-DDES when applied to matrices ``FDmesh2'' and ``FDmesh3''. For RF-KRYLOV, we report results
for all different values of $nev$ and number of MPI processes. For RF-DDES we only report the highest 
times across all different values of $nev,\ \tau$ and $p$. RF-DDES was found to spend a considerably
smaller amount of time on orthonormalization than what RF-KRYLOV did, mainly because $s$ was much 
smaller than $n$ (the values of $s$ for $p=2$ and $p=4$ can be found in Table \ref{testmat2}). Indeed, 
if both RF-KRYLOV and Algorithm \ref{alg:arnoldi} in RF-DDES perform a similar number of iterations, 
we expect the former to spend roughly $n/s$ more time on orthonormalization compared to RF-DDES.

\begin{figure}
\begin{tikzpicture}[scale=0.73]
\begin{axis}[
    ybar,
    bar width=0.30cm,
    enlarge x limits=0.15,
    enlarge y limits={0.15, upper},
    legend style={at={(0.5,-0.15)},
    anchor=north,legend columns=-1},
    ylabel={Time (s)},
    symbolic x coords={$p=2$,$p=4$,$p=8$,$p=16$,$p=32$},
    xtick=data,
    nodes near coords,
    nodes near coords align={vertical},
    ]
\addplot +[postaction={pattern=dots}] coordinates {($p=2$,157) ($p=4$,82) ($p=8$,56) ($p=16$,50) ($p=32$,49)};
\addplot +[postaction={pattern=horizontal lines}] coordinates {($p=2$,51) ($p=4$,21) ($p=8$,10) ($p=16$,6) ($p=32$,4)};
\addplot +[postaction={pattern=north east lines}] coordinates {($p=2$,213) ($p=4$,107) ($p=8$,72) ($p=16$,64) ($p=32$,62)};
\legend{Interface,Interior,Total}
\end{axis}
\end{tikzpicture}
\qquad
\begin{tikzpicture}[scale=0.73]
\begin{axis}[
    ybar,
    bar width=0.30cm,
    enlarge x limits=0.15,
    enlarge y limits={0.15, upper},
    legend style={at={(0.5,-0.15)},
      anchor=north,legend columns=-1},
    ylabel={Time (s)},
    symbolic x coords={$p=2$,$p=4$,$p=8$,$p=16$,$p=32$},
    xtick=data,
    nodes near coords,
    nodes near coords align={vertical},
    ]
\addplot +[postaction={pattern=dots}] coordinates {($p=2$,65) ($p=4$,45) ($p=8$,31) ($p=16$,26) ($p=32$,27)};
\addplot +[postaction={pattern=horizontal lines}] coordinates {($p=2$,48) ($p=4$,38) ($p=8$,28) ($p=16$,21) ($p=32$,15)};
\addplot +[postaction={pattern=north east lines}] coordinates {($p=2$,116) ($p=4$,85) ($p=8$,63) ($p=16$,53) ($p=32$,48)};
\legend{Interface,Interior,Total}
\end{axis}
\end{tikzpicture}
\caption{Amount of time required to apply the rational filter (``Interface''), form the subspace 
associated with the interior variables (``Interior''), and total wall-clock time (``Total'') 
obtained by an MPI-only execution of RF-DDES for the case where $nev=300$. Left: ``shipsec8''. 
Right: ``FDmesh2''.} \label{tikz1}
\end{figure}
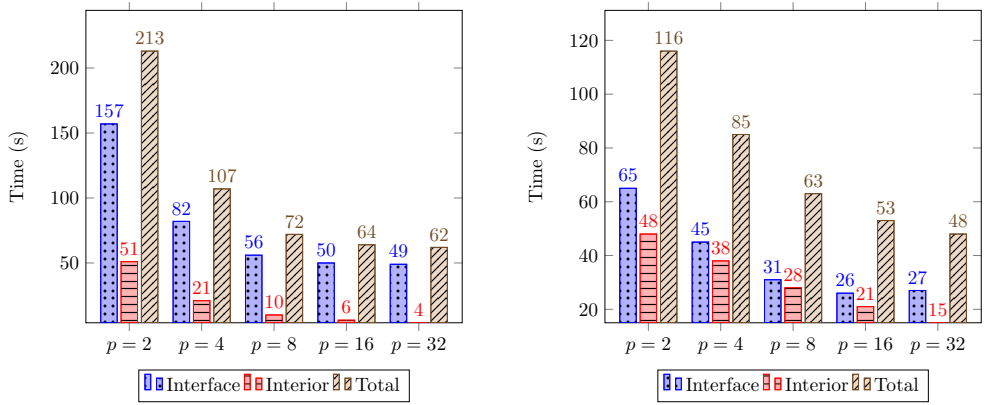
Figure \ref{tikz1} lists the wall-clock times achieved by an MPI-only implementation of 
RF-DDES, i.e., $p$ still denotes the number of subdomains but each subdomain is handled 
by a separate (single-threaded) MPI process, for matrices ``shipsec8'' and ``FDmesh2''.
In all cases, the MPI-only implementation of RF-DDES led to higher wall-clock times than 
those achieved by the hybrid implementations discussed in Tables \ref{tbl:19} and \ref{tbl:20}. 
More specifically, while the MPI-only implementation reduced the cost to construct and 
factorize the distributed $S_{\zeta_{\ell}}$ matrices, the application of the rational 
filter in Algorithm \ref{alg:arnoldi} became more expensive due to: a) each 
linear system solution with $S_{\zeta_{\ell}}$ required more time, b) a larger number of 
iterations had to be performed as $p$ increased (Algorithm \ref{alg:arnoldi} required 190, 
290, 300, 340 and 370 iterations for ``shipsec8'', and 160, 270, 320, 350, and 410 
iterations for ``FDmesh2'' as $p=2,\ 4,\ 8,\ 16$ and $p=32$, respectively). One more 
observation is that for the MPI-only version of RF-DDES, its scalability for increasing
values of $p$ is limited by the scalability of the linear system solver which is typically
not high. This suggests that reducing $p$ and applying RF-DDES recursively to the local 
pencils $(B_{\sigma}^{(j)},M_B^{(j)}),\ j=1,\ldots,p$ might be the best combination when 
only distributed memory parallelism is considered.

\section{Conclusion} \label{conclusion}

In this paper we proposed a rational filtering domain decomposition approach (termed as RF-DDES) for 
the computation of all eigenpairs of real symmetric pencils inside a given interval $[\alpha,\beta]$. 
In contrast with rational filtering Krylov approaches, RF-DDES applies the rational filter only to the 
interface variables. This has several advantages. First, orthogonalization is performed on vectors 
whose length is equal to the number of interface variables only. Second, the Krylov projection method 
may converge in fewer than $nev$ iterations. Third, it is possible to solve the original eigenvalue 
problem associated with the interior variables in real arithmetic and with trivial parallelism with 
respect to each subdomain. RF-DDES can be considerably faster than rational filtering Krylov approaches, 
especially when a large number of eigenvalues is located inside $[\alpha,\beta]$.

In future work, we aim to extend RF-DDES by taking advantage of additional levels of parallelism.
In addition to the ability to divide the initial interval $[\alpha,\beta]$ into non-overlapping 
subintervals and process them in parallel, e.g. see \cite{spectrumslicing,pfeast}, we can also 
assign linear system solutions associated with different quadrature nodes to different groups of 
processors. Another interesting direction is to consider the use of iterative solvers to solve the 
linear systems associated with $S(\zeta_1),\ldots,S(\zeta_{N_c})$. This could be helpful when 
RF-DDES is applied to the solution of symmetric eigenvalue problems arising from 3D domains. 
On the algorithmic side, it would be of interest to develop more efficient criteria to set the value 
of $nev_B^{(j)},\ j=1,\ldots,p$ in each subdomain, perhaps by adapting the work in \cite{doi:10.1137/040613767}. 
In a similar context, it would be interesting to also explore recursive implementations of RF-DDES. 
For example, RF-DDES could be applied individually to each matrix pencil $(B_\sigma^{(j)},M_B^{(j)}),\ 
j=1,\ldots,p$ to compute the $nev_B^{(j)}$ eigenvectors of interest. This could be particularly 
helpful when either $d_j$, the number of interior variables of the $j$th subdomain, or $nev_B^{(j)}$, 
are large.

\section{Acknowledgments}

Vassilis Kalantzis was partially supported by a Gerondelis Foundation Fellowship. The authors 
acknowledge the Minnesota Supercomputing Institute (MSI; http://www.msi.umn.edu) at the University 
of Minnesota for providing resources that contributed to the research results reported within 
this paper.

\bibliographystyle{siam} 
\bibliography{local2}

\end{document}